\documentclass[10pt,reqno]{amsart}
\pdfoutput=1

\usepackage{subfiles}

\usepackage{amsfonts,amsthm,amssymb,amsmath,amscd,euscript}
\usepackage{nccmath}
\usepackage{nicematrix}
\usepackage{mathdots}
\usepackage{csquotes}
\usepackage{mathrsfs}
\usepackage{bbm}
\usepackage{ifthen}
\usepackage{xpatch}
\usepackage{thmtools}
\usepackage{thm-restate}
\usepackage{calc}
\usepackage[reftex]{theoremref}

\usepackage[backend=biber,style=alphabetic,giveninits=true]{biblatex}
\usepackage[english]{babel}
\usepackage[hypertexnames=false]{hyperref}

\usepackage{multicol}
\usepackage[dvipsnames,rgb]{xcolor}
\usepackage{framed}
\usepackage{pdfpages}
\usepackage{scalerel}
\usepackage{fullpage}
\usepackage{color}
\usepackage{lipsum}
\usepackage{bold-extra}
\usepackage{fancyhdr}

\usepackage{tikz,tikz-3dplot}
\usepackage{tkz-berge}

\usepackage{pgfplots}
\pgfplotsset{compat=1.18}
\usetikzlibrary{shapes, arrows}
\usetikzlibrary{shapes.geometric}
\usetikzlibrary{graphs}
\usetikzlibrary{graphs.standard}
\usetikzlibrary{decorations.pathmorphing}
\usetikzlibrary{decorations.markings}
\tikzset{->-/.style={decoration={
  markings,
  mark=at position #1 with {\arrow{>}}},postaction={decorate}},
  ->-/.default=0.53
}
\tikzset{-<-/.style={decoration={
  markings,
  mark=at position #1 with {\arrow{<}}},postaction={decorate}},
  -<-/.default=0.53
}
\pgfplotsset{soldot/.style={color=blue,only marks,mark=*}} \pgfplotsset{holdot/.style={color=blue,fill=white,only marks,mark=*}}
\usepackage{tikz-cd}
\usepackage[all,2cell,cmtip]{xy}
\UseTwocells

\usepackage{mathtools}
\usepackage{float}
\usepackage[margin=1in,heightrounded]{geometry}
\usepackage{etoolbox}
\usepackage[most]{tcolorbox}
\tcbuselibrary{theorems}
\tcbuselibrary{skins}
\usepackage{mdframed}
\usepackage[inline, shortlabels]{enumitem}
\SetEnumitemKey{twocol}{itemsep = 1\itemsep,parsep  = 1\parsep,before  = \raggedcolumns\begin{multicols}{2},after   = \end{multicols}}
\usepackage{tasks}
\usepackage{stackengine}
\usepackage{tabularx}
\usepackage{soul}
\usepackage[a]{esvect}
\usepackage[activate={true,nocompatibility},nopatch=footnote,final,tracking=true,kerning=true,spacing=true,factor=1100,stretch=10,shrink=10]{microtype}
\microtypecontext{spacing=nonfrench}
\usepackage{siunitx}
\usepackage{mathrsfs}
\usepackage{diagbox}
\DeclareFontFamily{U}{mathx}{\hyphenchar\font45}
\DeclareFontShape{U}{mathx}{m}{n}{
      <5> <6> <7> <8> <9> <10>
      <10.95> <12> <14.4> <17.28> <20.74> <24.88>
      mathx10
      }{}
\DeclareSymbolFont{mathx}{U}{mathx}{m}{n}
\DeclareFontSubstitution{U}{mathx}{m}{n}
\DeclareMathAccent{\widecheck}{0}{mathx}{"71}
\DeclareMathAccent{\wideparen}{0}{mathx}{"75}

\usepackage{setspace}
\usepackage{stmaryrd}
\usepackage{centernot}
\usepackage{subfig}
\usepackage{cancel}
\newcommand\Ccancel[2][black]{
    \let\OldcancelColor\CancelColor
    \renewcommand\CancelColor{\color{#1}}
    \cancel{#2}
    \renewcommand\CancelColor{\OldcancelColor}
}
\usepackage{makecell}
\hypersetup{colorlinks=true,citecolor=blue,urlcolor =blue,linkbordercolor={1 0 0}}

\definecolor{bluepigment}{rgb}{0.2, 0.2, 0.6}
\definecolor{cobalt}{rgb}{0.0, 0.28, 0.67}
\definecolor{darkpowderblue}{rgb}{0.0, 0.2, 0.6}
\definecolor{egyptianblue}{rgb}{0.06, 0.2, 0.65}
\definecolor{lapislazuli}{rgb}{0.15, 0.38, 0.61}
\definecolor{tealblue}{rgb}{0.21, 0.46, 0.53}
\definecolor{darktealblue}{rgb}{0.11, 0.36, 0.43}
\definecolor{zaffre}{rgb}{0.0, 0.08, 0.66}
\definecolor{zaffreprime}{rgb}{0.4, 0.18, 0.76}
\definecolor{arsenic}{rgb}{0.23, 0.27, 0.29}
\definecolor{bulgarianrose}{rgb}{0.28, 0.02, 0.03}
\definecolor{aurometalsaurus}{rgb}{0.43, 0.5, 0.5}
\definecolor{amethyst}{rgb}{0.6, 0.4, 0.8}

\definecolor{pblue}{rgb}{.114, .102, .365}
\definecolor{pblue2}{rgb}{.204, .196, .471}
\definecolor{pblue3}{rgb}{.655,.784,.855}
\definecolor{pgreen}{rgb}{.075,.212,.337}
\definecolor{pgreen3}{rgb}{.667,.882,.765}

\definecolor{codegreen}{rgb}{0,0.6,0}
\definecolor{codegray}{rgb}{0.5,0.5,0.5}
\definecolor{codepurple}{rgb}{0.58,0,0.82}
\definecolor{backcolour}{rgb}{0.95,0.95,0.92}

\allowdisplaybreaks[1]

\makeatletter
\patchcmd{\@maketitle}%
  {\ifx\@empty\@dedicatory}%
  {\ifx\@empty\@date \else {\vskip3ex \centering\footnotesize\@date\par\vskip1ex}\fi%
   \ifx\@empty\@dedicatory}%
  {}{}%
\patchcmd{\@adminfootnotes}%
  {\ifx\@empty\@date\else \@footnotetext{\@setdate}\fi}%
  {}{}{}%

\def\section{\@startsection{section}{1}%
  \z@{.7\linespacing\@plus\linespacing}{.5\linespacing}%
  {\normalfont\Large\bfseries}}
\def\subsection{\@startsection{subsection}{2}%
  \z@{.5\linespacing\@plus.7\linespacing}{.5\linespacing}%
  {\normalfont\large\bfseries}}%
\def\subsubsection{\@startsection{subsubsection}{3}%
  \z@{.5\linespacing\@plus.7\linespacing}{.5\linespacing}%
  {\normalfont\bfseries}}%
\def\@tocline#1#2#3#4#5#6#7{\relax%
  \ifnum #1>\c@tocdepth 
  \else%
    \par \addpenalty\@secpenalty\addvspace{#2}%
    \begingroup \hyphenpenalty\@M%
    \@ifempty{#4}{%
      \@tempdima\csname r@tocindent\number#1\endcsname\relax%
    }{%
      \@tempdima#4\relax%
    }%
    \parindent\z@ \leftskip#3\relax \advance\leftskip\@tempdima\relax%
    \rightskip\@pnumwidth plus4em \parfillskip-\@pnumwidth%
    #5\leavevmode\hskip-\@tempdima%
      \ifcase #1%
       \or\or \hskip 1em \or \hskip 2em \else \hskip 3em \fi%
      #6\nobreak\relax%
    \dotfill\hbox to\@pnumwidth{\@tocpagenum{#7}}\par%
    \nobreak%
    \endgroup%
  \fi}%
\makeatother%

\makeatletter%
\xpatchcmd{\endmdframed}%
    {\aftergroup\endmdf@trivlist\color@endgroup}%
    {\endmdf@trivlist\color@endgroup\@doendpe}%
    {}{}%
\makeatother%

\newenvironment{solution} {\begin{mdframed}[skipabove=0.1in,skipbelow=0.1in]\begin{proof}[Solution]} {\end{proof}\end{mdframed}}

\mdfdefinestyle{theoremstyle}{fontcolor=amethyst, linecolor=amethyst, bottomline=false,topline=false,rightline=false,skipabove=2pt,skipbelow=2pt}

\declaretheoremstyle[
spaceabove={5pt},   
spacebelow={5pt},   
bodyfont=\itshape,  
headfont=\bfseries, 
headpunct={.},         
postheadspace={.5em}      
]{theorem}
\theoremstyle{theorem}
\declaretheorem[numberwithin=section, style=theorem, name=Theorem]{theorem}
\declaretheorem[numberwithin=section, style=theorem, name=Conjecture,sibling=theorem]{conjecture}
\declaretheorem[numberwithin=section, style=theorem, name=Question,sibling=theorem]{question}
\declaretheorem[ style=theorem, name=Lemma, sibling=theorem]{lemma}

\declaretheorem[sibling=theorem, style=theorem, name=Proposition]{proposition}
\declaretheorem[numberwithin=section, style=theorem, name=Definition]{defn}

\declaretheorem[name=Remark,sibling=theorem,style=definition]{remark}

\declaretheorem[sibling=theorem, style=theorem, name=Example]{example}

\declaretheorem[sibling={example}, style=theorem, name=Exercise]{exercise}

\declaretheorem[sibling={example}, style=theorem, name=Problem]{prob}

\declaretheorem[numbered=no, style=theorem, name=Problem]{prob-nonum}

\lstdefinestyle{mystyle}{
    keywordstyle=\color[rgb]{0,0,1},
    commentstyle=\color[rgb]{0.026,0.412,0.595},
    stringstyle=\color[rgb]{0.627,0.126,0.941},
    numberstyle=\color[rgb]{0.205, 0.142, 0.73},
    frame=single,
    breakatwhitespace=false,
    breaklines=true,
    postbreak=\mbox{\textcolor{darkgray}{$\hookrightarrow$}\space},
    captionpos=b,
    keepspaces=true,
    numbers=left,
    numbersep=5pt,
    showspaces=false,
    showstringspaces=false,
    showtabs=false,
    tabsize=2,
    basicstyle=\ttfamily,
    columns=fullflexible,
    keepspaces=true,
    language=C++
}

\DeclareRobustCommand{\SkipTocEntry}[3]{}

\usepackage{cleveref}

\crefname{defn}{definition}{definitions}
\Crefname{defn}{Definition}{Definitions}
\crefname{theorem}{theorem}{theorems}
\Crefname{theorem}{Theorem}{Theorems}
\crefname{prob}{problem}{problems}
\Crefname{prob}{Problem}{Problems}
\crefname{exercise}{exercise}{exercises}
\Crefname{exercise}{Exercise}{Exercises}
\crefname{proposition}{proposition}{propositions}
\Crefname{proposition}{Proposition}{Propositions}
\crefname{question}{question}{questions}
\Crefname{question}{Question}{Questions}

\makeatletter
\newcommand{\leqnomode}{\tagsleft@true\let\veqno\@@leqno}
\makeatother

\newcommand{\ZZ}{\mathbb{Z}}
\newcommand{\R}{\mathbb{R}}
\newcommand{\Q}{\mathbb{Q}}
\newcommand{\N}{\mathbb{N}}
\newcommand{\CC}{\mathbb{C}}

\newcommand{\abs}[1]{\left\lvert #1 \right\rvert}

\newcommand{\0}{\emptyset}

\renewcommand*{\st}{\;\vert\;}

\DeclareMathOperator{\Id}{Id}

\DeclareMathOperator{\im}{im}

\DeclareMathOperator{\Hom}{Hom}

\newcommand{\rest}[2]{{
  \left.\kern-\nulldelimiterspace 
  #1 
  \vphantom{\big|} 
  \right|_{#2} 
  }}

\newcommand{\floor}[1]{\left\lfloor#1\right\rfloor}

\DeclareMathOperator{\GL}{GL}
\DeclareMathOperator{\SL}{SL}
\DeclareMathOperator{\PSL}{PSL}
\DeclareMathOperator{\Sp}{Sp}
\DeclareMathOperator{\tr}{tr}

\newcommand{\e}{\varepsilon}
\newcommand{\de}{\delta}

\DeclareMathOperator{\Mod}{Mod}
\DeclareMathOperator{\PMod}{PMod}

\DeclareMathOperator{\Stab}{Stab}

\DeclareMathOperator{\Diff}{Diff}

\newcommand*{\PP}{\mathbb{P}}

\newcounter{proofstep}
\newcounter{proofsubstep}[proofstep]
\newcommand{\step}[1]{%
  \refstepcounter{proofstep}
  \par\medskip\noindent
  \textbf{Step \theproofstep: #1.} \enspace
}

\usepackage{etoolbox}
\AtBeginEnvironment{proof}{\setcounter{proofstep}{0}}
\usepackage{forest}

\usepackage{subfig}
\usepackage{bm}

\DeclareMathOperator{\Bl}{B\ell}
\DeclareMathOperator{\Br}{Br}
\DeclareMathOperator{\Perm}{Perm}
\DeclareMathOperator{\PSp}{PSp}
\DeclareMathOperator{\PGL}{PGL}

\DeclareMathOperator{\Conf}{Conf}

\newcommand{\para}[1]{\paragraph{\color{Black}\textbf{#1}}}

\usepackage[mode=image]{standalone}

\DeclareMathOperator{\Core}{Core}

\graphicspath{ {resources} }

\newcommand{\VARTitle}{Universal braids for elliptic fibrations}
\newcommand{\VARSubtitle}{from character varieties to Coxeter's factor groups} 
\newcommand{\VARAuthor}{Faye Jackson}
\author{\vspace{-0.25cm}\VARAuthor}

\address{Department of Mathematics, University of Chicago, 5734 S. University Ave, Chicago, IL 60637}
\email{alephnil@uchicago.edu}
\thanks{This material is based upon work supported by the National Science Foundation Graduate Research Fellowship Program under Grant No. 2140001. Any opinions, findings, and conclusions or recommendations expressed in this material are those of the author and do not necessarily reflect the views of the National Science Foundation.}

\title{\VARTitle: \\ \VARSubtitle} 
\date{}
\definecolor{darkpastelpurple}{rgb}{0.59, 0.44, 0.84}
\definecolor{mauve}{rgb}{0.88, 0.69, 1.0}

\begin{document}

\hypersetup{linkcolor=lapislazuli}

\begin{abstract} 
	Let $\pi : M \to B$ be an elliptic fibration over $B = D^2$ or $B = S^2$ with $n$ nodal fibers over $\Delta \subseteq B$. We study the \textit{universal liftable braids} for $\pi$: those braids that admit a fiber-preserving lift to $M$ for all choices of coordinates on $(B,\Delta)$. When $B = S^2$, we show that nontrivial universal braids do not exist by proving a Zariski-density theorem on the $\SL_2$-character variety for $(S^2,\Delta)$. When $B = D^2$ we classify when the subgroup of universal braids has finite index in the braid group $B_n = \Mod(D^2,\Delta)$, and relate these examples to Coxeter's factor groups of braid groups, which in turn are related to the platonic solids. Finally, we generalize the results derived in the finite-index cases by considering a canonical family of branched covers of the base $B$ associated to any elliptic fibration. The generalization naturally connects the universal braids to the integral Burau representation reduced modulo 3.
\end{abstract}

\thispagestyle{empty}
\maketitle
\setcounter{tocdepth}{2}
\microtypesetup{protrusion=false}
\microtypesetup{protrusion=true}

\lhead{\VARAuthor}
\rhead{\VARTitle}
\chead{}
\setstretch{1.2}

\tableofcontents

\section{Introduction}\label{sec:intro}

Let $\pi : M \to B$ be a relatively minimal smooth elliptic fibration over $B = S^2$ or $B = D^2$ with $n$ nodal fibers located over $\Delta \subseteq B$. Moishezon, utilizing algebraic input from Livne, classified such fibrations up to smooth isomorphism for $B = S^2$ and showed that they are all pullbacks of the rational elliptic surface $\Bl_{\{p_1,\ldots,p_9\}} \PP^2 \to \PP^1$ along $z \mapsto z^d$ for some $d \geq 1$ \cite[Theorem 9]{moishezon}. Associated to $\pi$ is its smooth automorphism group $\Mod(\pi)$ consisting of fiber-preserving diffeomorphisms of $M$ up to fiber-preserving isotopy as well as the subgroup $\Br(\pi)$ of braids\footnote{When $B = S^2$, the mapping class group $\Mod(B,\Delta)$ differs slightly from the spherical braid group $\pi_1(\Conf_n(S^2))$. One must quotient $\pi_1(\Conf_n(S^2))$ by its center, which is a $\ZZ/2\ZZ$-subgroup \cite[Section 9.1]{primer}.} admitting a lift to a fiber-preserving diffeomorphism of the 4-manifold $M$ that restricts to the identity on $\partial M$. In \cite{how-large,torsion-braid-mon}, we studied the index $[\Mod(B,\Delta) : \Br(\pi)]$ of liftable braids as well as the finite order liftable braids up to conjugacy in $\Br(\pi)$ when $B = S^2$. Notably, $\Br(\pi)$ is not a normal subgroup of $\Mod(B,\Delta)$. One can interpret conjugation in this setting as a change of coordinates in the base $B$, and under this interpretation $\Br(\pi)$ is not invariant under change of coordinates on $(B,\Delta)$. In this paper, we study the maximal coordinate-invariant subgroup of $\Br(\pi)$: the \textit{normal core} of $\Br(\pi)$ in the mapping class group $\Mod(B,\Delta) \coloneqq \pi_0(\Diff^+(B,\Delta))$ defined by
\begin{align*}
	\Core(\Br(\pi)) \coloneqq  \bigcap_{g \in \Mod(B,\Delta)} g\Br(\pi)g^{-1}.
\end{align*}
We refer to elements of $\Core(\Br(\pi))$ as \textit{universal braids} associated to $\pi$, since these braids are liftable regardless of the choice of coordinates on $(B,\Delta)$. The group $\Core(\Br(\pi))$ differs substantially depending on whether the base $B$ is a sphere or a disk. In particular,
\begin{enumerate}
	\item[1.] Over $B = S^2$, there are no universal braids (see \Cref{thm:normal-core-triv}), i.e., $\Core(\Br(\pi)) = 1$.
	\item[2a.] Over $B = D^2$, we classify when $\Br(\pi)$---and hence $\Core(\Br(\pi))$---has finite index in the braid group $B_n \coloneqq \Mod(D^2,\Delta)$. We say that these fibrations have \textit{virtually universal} braid groups.
	\item[2b.] Furthermore when $B = D^2$, in the nontrivial finite-index cases, the group of universal braids has a concrete description: 
		\[
			\Core(\Br(\pi)) = \langle \langle \sigma_1^3 \rangle  \rangle,
		\]
		where $\sigma_1$ is a half-twist and $\langle \langle \sigma_1^3 \rangle  \rangle$ is the normal subgroup generated by $\sigma_1^3$. This calculation relates $\Core(\Br(\pi))$ to the three-dimensional platonic solids via Coxeter's factor groups of braid groups (see \Cref{rem:coxeter-platonic}).
\end{enumerate}
To generalize the special cases described in (2b) to genus one fibrations over the disk with arbitrary configurations of nodal fibers, we introduce a canonical way of associating a branched cover of the base Riemann surface to a genus $g$ Lefschetz fibration, thereby reducing the dimensionality of the problem. In the genus one case, this construction produces a branched cover $\pi[N] : \Sigma_\pi[N] \to B$ of Riemann surfaces for each $N \geq 1$. Furthermore, when $N = 2$, the subgroup of liftable braids associated to this branched cover is closely related to the integral Burau representation reduced modulo 3. For precise definitions and theorem statements, see \Cref{thm:symp-iso-intro} and the discussion preceding it.

\vspace{0.2cm}
\para{Universal Braids over the sphere}

Our first result is when the base space $B$ is the sphere $S^2$. Throughout this paper, when we refer to a smooth genus one fibration, we assume that all of the fibers are nodal (i.e., all singularities are simple).
\begin{theorem}[Universal braids over $S^2$ are trivial]\label{thm:normal-core-triv}
	Let $\pi : M \to S^2$ be a nontrivial smooth genus one Lefschetz fibration. Then $\Core(\Br(\pi))$ is trivial. In other words, there are no universal braids over $S^2$.
\end{theorem}
\noindent To prove \Cref{thm:normal-core-triv}, we analyze a particular mapping class group orbit on the $\SL_2$-character variety
\begin{align*}
	\mathfrak{X}_\CC(S^2 \setminus \Delta) &\coloneqq \Hom(\pi_1(S^2 \setminus \Delta), \SL_2\CC)\sslash \SL_2\CC,
\end{align*}
parameterizing representations of $\pi_1(S^2 \setminus \Delta)$ into $\SL_2\CC$. The mapping class group $\Mod(S^2,\Delta)$ acts on $\pi_1(S^2 \setminus \Delta)$ through outer automorphisms, inducing an action on $\mathfrak{X}_\CC(S^2 \setminus \Delta)$ referred to as the \textit{Hurwitz action}. Two representations $\rho_1,\rho_2 : \pi_1(S^2 \setminus \Delta) \to \SL_2\CC$ are called \textit{Hurwitz equivalent} if they are related by the Hurwitz action and conjugacy in $\SL_2\CC$. Fix an ordering $p_1,\ldots,p_n$ of the points $p_i \in \Delta$. The geometry of $\mathfrak{X}_\CC(S^2 \setminus \Delta)$ can be probed by the map
\begin{align*}
	\tr_\Delta : \mathfrak{X}_\CC(S^2 \setminus \Delta) &\to \mathbb{A}^n \\ 
	[\rho] &\mapsto (\tr(\rho(\gamma_i)))_{i=1}^n,
\end{align*}
where $n = \abs{\Delta}$ is the number of singular points, $\gamma_i$ is some simple loop enclosing the puncture $p_i$, and $[\rho]$ is the conjugacy class of some representation $\rho$. The \textit{relative character varieties}
\begin{align*}
	\mathfrak{X}_{\vec{k},\CC}(S^2 \setminus \Delta) &\coloneqq \tr_\Delta^{-1}(\vec{k})
\end{align*}
for $\vec{k} \in \mathbb{A}^n$ are preserved by the Hurwitz action of the pure mapping class group $\PMod(S^2,\Delta)$. In order to prove \Cref{thm:normal-core-triv}, we first prove the following density statement.
\begin{theorem}[Zariski density of Hurwitz orbits]\label{thm:z-dense}
	Let $\pi : M \to S^2$ be a nontrivial smooth genus one Lefschetz fibration with monodromy representation $\phi_\pi : \pi_1(S^2 \setminus \Delta) \to \SL_2\ZZ$. Then, with $\vec{k} = (2,\ldots,2)$, the orbit $\Mod(S^2,\Delta) \cdot [\phi_\pi]$ is Zariski dense in $\mathfrak{X}_{\vec{k},\CC}(S^2 \setminus \Delta)$.
\end{theorem}

\begin{remark}
	Throughout this paper, all smooth genus one fibrations will be assumed to have nodal fibers and to be relatively minimal. Furthermore the monodromy representation $\phi_\pi$ is defined as a \textit{homomorphism} which sends counterclockwise loops about singular fibers to \textit{left-handed} Dehn twists. This convention stands in contrast to the literature on Lefschetz fibrations, but agrees with the conventions in the rest of geometric topology. It can be recovered from the usual conventions by an inversion (see \cite[p. 291]{gompf-stipsicz}).
\end{remark}

\noindent To prove \Cref{thm:z-dense} we verify a condition developed by Coccia--Litt \cite[Proposition 3.1.10]{coccia-litt}. For the reader's convenience, we outline here the proof that \Cref{thm:z-dense} implies \Cref{thm:normal-core-triv} before carrying out the formal proof in \Cref{sec:normal-core-triv}. Note the following observations:
\begin{enumerate}
	\item $\mathfrak{X}_{\vec{k},\CC}(S^2 \setminus \Delta)$ contains an equivariantly embedded copy of the Teichm\"{u}ller space $\mathcal{T}_{0,n}$, and so $\PMod(S^2,\Delta)$ acts faithfully on $\mathfrak{X}_{\vec{k},\CC}(S^2 \setminus \Delta)$.
	\item $\Br(\pi)$ is given by the stabilizer $\Stab\, [\phi_\pi]$ under the Hurwitz action by a theorem of Moishezon (see \cite[Theorem 2.1]{how-large}), and so
\begin{align}
	\Core(\Br(\pi)) = \bigcap_{g \in \Mod(S^2,\Delta)} \Stab (g \cdot [\phi_\pi]).
\end{align}
\end{enumerate}
Thus, \Cref{thm:z-dense} implies that $\Core(\Br(\pi))$ injects into the symmetric group $S_n$ since the action of $\Mod(S^2,\Delta)$ is algebraic. \Cref{thm:normal-core-triv} follows via standard mapping class group theory which prohibits the existence of finite normal subgroups of $\Mod(S^2,\Delta)$ when $\abs{\Delta} \geq 5$ (see \cite[Sections 11.5 and 11.6]{ivanov-teich-groups}).

\vspace{0.2cm}
\para{Universal braids over the disk}

Let $\pi : M \to D^2$ be a smooth genus one Lefschetz fibration with $n$ nodal fibers over $\Delta \subseteq D^2 \setminus \partial D^2$. After fixing some base point $b \in \partial D^2$ one may consider the \textit{representation variety}
\begin{align*}
	\mathcal{R}_\CC(D^2,\Delta) \coloneqq \Hom(\pi_1(D^2 \setminus \Delta,b), \SL_2\CC).
\end{align*}
As above, $\mathcal{R}_\CC(D^2,\Delta)$ admits an action by $\Mod(D^2,\Delta)$ by precomposition (with an inversion to make it a left action). Note that the classical braid group $B_n$ equals $\Mod(D^2, \Delta)$, since braids are naturally mapping classes on the disk \cite[Chapter 9]{primer}. In a slight abuse of notation, we also refer to the action of $B_n$ on $\mathcal{R}_\CC(D^2,\Delta)$ as the Hurwitz action. Our next result is a classification of when $\Br(\pi)$ has finite index in $B_n$ and a calculation of $\Core(\Br(\pi))$ in these cases. Examples of this phenomenon appear naturally as small restrictions of the monodromy representation of the rational elliptic fibration over $\PP^1$ to embedded disks (see \Cref{sec:finite-stuff} for more details). These examples were also independently studied by L\"{o}nne and Ito \cite{lonne-br3,ito-fin-braid-hurwitz}. 
\begin{theorem}[Classification of virtually universal braid groups]\label{thm:finite-classification}
	Let $\pi : M \to D^2$ be a smooth genus one Lefschetz fibration with $n \geq 2$ nodal fibers over $\Delta = \{x_1,\ldots,x_n\}$ and monodromy representation 
	\[
		\phi_\pi : \pi_1(D^2 \setminus \Delta) \to \SL_2\ZZ,
	\]
	then
	\begin{enumerate}[label=(\arabic*),ref=(\arabic*)]
		\item\label{item:virt-univ-class} the index $[B_n : \Br(\pi)] < \infty$ if and only if $\phi_\pi$ is Hurwitz equivalent to one of $\phi_{n}$ or $\psi_n$ where
	\begin{align*}
		\phi_n(\gamma_i) &\coloneqq {\begin{pmatrix} 1 & -1 \\ 0 & 1 \end{pmatrix}}	\text{ if } i \in \{1,\ldots,n\} \text{ is odd} && \psi_n(\gamma_i) \coloneqq {\begin{pmatrix} 1 & -1 \\0 & 1\end{pmatrix}} \text{ for all } i \in \{1,\ldots,n\} \\
		\phi_n(\gamma_i) &\coloneqq {\begin{pmatrix} 1 & 0 \\ 1 & 1 \end{pmatrix}}	\text{ if } i \in \{1,\ldots,n\} \text{ is even} 
	\end{align*}
	for $\gamma_1,\ldots,\gamma_n$ a standard sequence of generators for $\pi_1(D^2 \setminus \Delta)$. Furthermore, when $\phi_\pi$ is Hurwitz equivalent to $\phi_n$, then $n \leq 4$.
		\item\label{item:virt-univ-cox} provided that $\phi_\pi$ is Hurwitz equivalent to $\phi_n$, the group $\Core(\Br(\pi))$ is
			\begin{align*}
				\Core(\Br(\pi)) = \langle \langle \sigma_1^3 \rangle \rangle,
			\end{align*}
			where $\sigma_1$ is a single half-twist in $B_n$ and $\langle \langle \sigma_1^3 \rangle  \rangle$ indicates the normal closure of $\langle \sigma_1^3 \rangle $. Note that, since all half-twists are conjugate in $B_n$, the subgroup $\langle \langle \sigma_1^3 \rangle  \rangle$ is independent of the choice of half-twist $\sigma_1$.
	\end{enumerate}
\end{theorem}
\noindent 
Note that part \ref{item:virt-univ-class} of \Cref{thm:finite-classification} is equivalent to the existence of an isomorphism 
\begin{center}
	\begin{tikzcd}
		M \ar[r,"\widetilde{f}"] \ar[d,"\pi"'] & M' \ar[d,"\pi'"] \\
		D^2 \ar[r,"f"] & D^2
	\end{tikzcd}
\end{center}
such that $f\big|_{\partial D^2} = \Id$ between $\pi : M \to D^2$ and a standard $\pi' : M' \to D^2$ (specified by $\psi_n$ or $\phi_n$) whenever $[B_n : \Br(\pi)] < \infty$ \cite[Lemma 7a]{moishezon}. The proof of part \ref{item:virt-univ-class} of \Cref{thm:finite-classification} uses a lemma of the author \cite[Lemma 5.1]{how-large} to restrict the vanishing cycles of $\pi$ to curves which intersect at most once. This technique reduces the classification to a combinatorial search on the Farey complex to find vanishing cycles with intersection number $\geq 2$ in the Hurwitz orbit of $\phi_\pi$ for some finite list of fibrations $\pi$. The proof of part \ref{item:virt-univ-cox} is computational, and relies on Coxeter's formula for $B_n/\langle \langle \sigma_1^3 \rangle  \rangle$ when $n \in \{3,4\}$ (see \Cref{thm:coxeter}). We defer to \Cref{subsec:finite} and \Cref{rem:coxeter-platonic} for a discussion of the relationship between part \ref{item:virt-univ-cox} of \Cref{thm:finite-classification}, Coxeter's factor groups of braid groups, and the platonic solids.
\begin{remark}\label{rmk:ito-lonne}
	Part \ref{item:virt-univ-class} of \Cref{thm:finite-classification} can be deduced from \cite[Theorem 3]{ito-fin-braid-hurwitz} along with a finite computation. Ito shows that if $n \geq 5$, then for any representation $\widetilde{
\rho} : F_n \to B_3$ with finite Hurwitz orbit, there is a partition $I,J$ of the generators $\gamma_i$ of $F_n$ so that $\widetilde{\rho}(\gamma_i)$ commutes with $\widetilde{\rho}(\gamma_j)$ for $i \in I, j \in J$. Lifting a representation $\rho : F_n \to \SL_2\ZZ \cong B_3/c^2$ where $Z(B_3) = \langle c \rangle$ shows that two lifts of a Dehn twist commute if and only if they differ by $Z(B_3)$, and so $\rho$ is conjugate to $\psi_n$. For completeness, we give an independent proof of part \ref{item:virt-univ-class} of \Cref{thm:finite-classification} not relying on Ito's work.
	
Furthermore, L\"{o}nne's work implies that $\Br(\pi)$ has infinite index in $B_n$ for genus one fibrations with monodromy $\phi_n$ when $n \geq 5$; his results give the cases $n = 5, 6$ \cite[Lemma 3.2, Proposition 3.4, Theorem 3.5]{lonne-br3}, and the remaining cases follow by restricting to an embedded subdisk. The author independently showed this result in \cite[Theorem 1.3]{how-large}. L\"{o}nne also computes a generating set for $\Br(\pi)$ for the family $\phi_n$ when $n \leq 6$ \cite[Lemma 2.3, Theorem 3.5]{lonne-br3}.
\end{remark}

\para{Associating branched covers to Lefschetz fibrations}

The following construction arose as an attempt to understand the phenomena observed in part \ref{item:virt-univ-cox} of \Cref{thm:finite-classification} via branched covers. Let $\pi : M \to B$ be a Lefschetz fibration of genus $g$ over $B = S^2$ or $B = D^2$ with singular fibers over $\Delta \subseteq B$ and monodromy $\phi_\pi$. Given a finite-index normal subgroup $H$ of $\Mod(\Sigma_g)$, we construct a finite branched cover $\pi_H : \Sigma_\pi(H) \to B$ of the base. Informally, the quotient map $\Mod(\Sigma_g) \to \Mod(\Sigma_g)/H$ allows us to consistently associate to each fiber a finite set of data as in \Cref{fig:branched-from-lefschetz}. Formally, one constructs $\pi_H$ by taking the Galois cover with monodromy
\begin{align*}
	\phi_{\pi_H} : \pi_1(B \setminus \Delta) \xrightarrow{\phi_\pi} \Mod(\Sigma_g) \to \Mod(\Sigma_g)/H.
\end{align*}
Likewise, we obtain a finite-index subgroup $\Br(\pi_H) < \Mod(B,\Delta)$ consisting of those braids which lift to a fiber-preserving diffeomorphism of $\Sigma_\pi(H)$. Note that the reduction in dimension from $M$ to $\Sigma_\pi(H)$ leads to a corresponding reduction in complexity of the subgroup $\Br(\pi_H)$. \Cref{fig:branched-from-lefschetz} schematically depicts the procedure $\pi \rightsquigarrow \pi_H$ and the reduction in complexity in a particular case.
\begin{figure}
	\centering
	\includestandalone[scale=0.45]{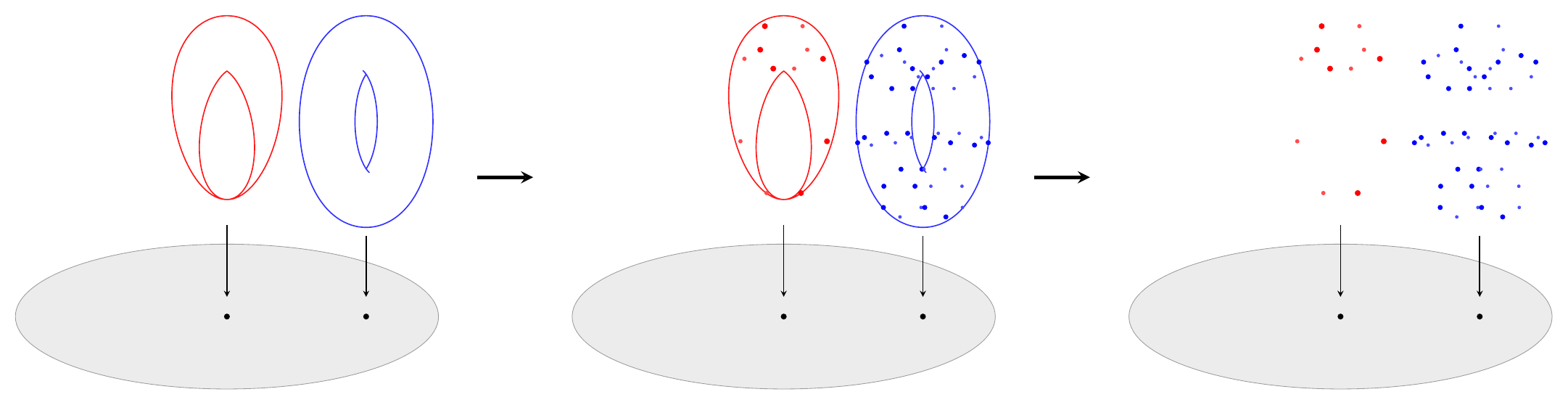}
	\caption{Constructing a branched cover from a genus one fibration. The displayed cover $p$ places the 7-torsion points (besides the origin) of the fiber over each regular value in the base. Note that this is equivalent to placing $H^1(F_x,\mathbb{F}_7) \setminus \{0\}$ over each $x \in D^2 \setminus 0$ where $F_x$ is the fiber over $x$. Thus the Galois closure of $p$ is $\pi[7] : \Sigma_\pi[7] \to D^2 \setminus 0$.}
	\label{fig:branched-from-lefschetz}
\end{figure}

For $g = 1$ define the \textit{level $N$ liftable braid group associated to $\pi$} by $\Br(\pi[N])$, where 
\[
	\pi[N] = \pi_{\Gamma(N)} : \Sigma_\pi[N] \to B
\]
and $\Gamma(N) = \ker(\SL_2\ZZ \to \SL_2(\ZZ/N\ZZ))$. An alternative viewpoint on $\pi[N]$ comes from the classifying map $B \setminus \Delta \to \mathbb{H}^2/\SL_2\ZZ$ of the fibration $\pi$ (see \cite{smith-lefschetz-fibrations} for a construction in the smooth category). Using this language, $\pi[N]$ is the pullback of the congruence modular curve $X(N) \coloneqq \mathbb{H}^2/\Gamma(N)$ over $X(1) \coloneqq \mathbb{H}^2/\SL_2\ZZ$:
\begin{center}
	\begin{tikzcd}
		\Sigma_\pi[N] \setminus (\pi[N])^{-1}(\Delta) \ar[d,"{\pi[N]}"'] \ar[r] & X(N) \ar[d,"\SL_2(\ZZ/N\ZZ)"] \\ 
		B \setminus \Delta \ar[r] & X(1)
	\end{tikzcd}
\end{center}
In the case when $B = D^2$, since $\bigcap_{N \in I} \Gamma(N) = 1$ for any unbounded set $I$ of natural numbers, observe that
\begin{align*}
	\Br(\pi) = \bigcap_{N \in I} \Br(\pi[N]),
\end{align*}
and so in a precise sense $\Br(\pi[N])$ approximates $\Br(\pi)$ more and more accurately as $N$ increases. See \Cref{prop:local-to-global-smooth} for a more precise and a more general statement. Furthermore, a direct computation shows that when $B = D^2$ and $\phi_\pi$ is Hurwitz equivalent to $\phi_2,\phi_3,$ or $\phi_4$ we have that $\Br(\pi) = \Br(\pi[2])$. These observations motivate the following theorem, which generalizes part \ref{item:virt-univ-cox} of \Cref{thm:finite-classification}. Before stating the theorem, we require some basic definitions regarding symplectic groups and hyperelliptic covers. 

Given a disk $D^2$ with $n$ marked points, there is a double branched cover $\Sigma \to D^2$ branched over those $n$ points, referred to as the \textit{hyperelliptic cover} of $D^2$ with $n$ branch points. The surface $\Sigma$ has genus $g = \frac{n-1}{2}$ with one boundary component when $n$ is odd and it has genus $g = \frac{n-2}{2}$ with two boundary components when $n$ is even. Now given a finite field $\mathbb{F}_p$ with $p \neq 2$, and a vector $v \in \mathbb{F}_p^{2\ell}$, consider the subgroup $(\Sp_{2\ell}(\mathbb{F}_p))_v < \Sp_{2\ell}(\mathbb{F}_p)$ consisting of all symplectic matrices fixing $v$. The stabilizer $(\Sp_{2\ell}(\mathbb{F}_p))_v$ fits into an exact sequence
\begin{align*}
	1 \to U \to (\Sp_{2\ell}(\mathbb{F}_p))_v \to \Sp(v^\perp/v) \to 1,
\end{align*}
where $U$ is referred to as the unipotent radical of the subgroup $(\Sp_{2\ell}(\mathbb{F}_p))_v$ (see \Cref{prop:levi} for more details). We let $Z(U)$ denote the center of $U$, and note that $Z(U) \cong \mathbb{F}_p$. Equipped with this notation, we are prepared to state the theorem.

\begin{theorem}[Level two congruence core quotients]\label{thm:symp-iso-intro}
	Let $\pi : M \to D^2$ be a genus one Lefschetz fibration with $n$ singular fibers. For convenience let $g$ be the genus of the hyperelliptic cover of $D^2$ with $n$ branch points. Using the notation above:
	\begin{enumerate}
		\item $\Br(\pi[2]) = B_n$. In this case, the image of the monodromy $\phi_{\pi[2]} : \pi_1(D^2 \setminus \Delta) \to \SL_2 \mathbb{F}_2$ lies in a $\ZZ/2\ZZ$ subgroup of $\SL_2 \mathbb{F}_2 \cong S_3$.
		\item $\Br(\pi[2]) \neq B_n$ and the monodromy $\phi_{\pi[2]}(\partial D^2)$ of $\Sigma_\pi[2] \to D^2$ about $\partial D^2$ is nontrivial. In this case,
			\begin{align*}
				B_n/\Core(\Br(\pi[2])) &\cong
				\begin{cases}
					(\Sp_{2g+2}(\mathbb{F}_3))_{y_{g+1}} < \Sp_{2g+2}(\mathbb{F}_3) & n \equiv 0 \pmod{2} \\
					\Sp_{2g}(\mathbb{F}_3) & n \equiv 1 \pmod{2}
				\end{cases},
			\end{align*}
			where $y_{g+1} \in \mathbb{F}_3^{2g+2}$ is a fixed nonzero vector.
		\item $\Br(\pi[2]) \neq B_n$ and the monodromy $\phi_{\pi[2]}(\partial D^2)$ of $\Sigma_\pi[2] \to D^2$ about $\partial D^2$ is trivial, so that $n$ is necessarily even. In this case,
			\begin{align*}
				B_n/\Core(\Br(\pi[2])) \cong (\Sp_{2g+2}(\mathbb{F}_3))_{y_{g+1}}/Z(U).
			\end{align*}
	\end{enumerate}
	Furthermore, the isomorphisms in the latter two cases are naturally realized by the integral Burau representation reduced modulo $3$.
\end{theorem}
\noindent The decomposition of $\SL_2 \mathbb{F}_2 \cong S_3$ as a semidirect product $\ZZ/3\ZZ \rtimes \ZZ/2\ZZ$ explains the appearance of the Burau representation as well as the $\ZZ/3\ZZ$-homology in \Cref{thm:symp-iso-intro}, as we explain in \Cref{sec:branched-covers}. Notably, the cover $\Sigma_\pi[2]$ decomposes as a sequence of covers $\Sigma_\pi[2] \to S[2] \to D^2$, where $S[2]$ is the hyperelliptic surface. The lift from $S[2]$ to $\Sigma_\pi[2]$ is controlled by a mod 3 cohomology class, explaining the appearance of the Burau representation.

\begin{remark}[\textbf{$\pi_1$ of certain Hurwitz spaces}]
	\Cref{thm:symp-iso-intro} can be understood as a computation of the orbifold fundamental group of certain components of Hurwitz spaces over the configuration spaces $\Conf_n(\operatorname{Int}(D^2))\cong \Conf_n(\CC)$. In particular, let $\mathcal{H}_n$ denote the Hurwitz space of normal (not necessarily connected) branched covers of the disk over $n$ points in the interior such that
	\begin{enumerate}
		\item the Galois group of the cover is contained in $S_3$,
		\item the branching at each puncture is simple, i.e., the monodromy about a puncture is given by a transposition,
		\item each branched cover $p : X \to D^2$ is equipped with an identification $p^{-1}(b)$ with $S_3$, where $b \in \partial D^2$ is fixed.
	\end{enumerate}
	There are four types of components of $\mathcal{H}_n$. Case (1) of \Cref{thm:symp-iso-intro} corresponds to the components where $p$ is a disconnected cover. The other three types of components are determined by the conjugacy class of the monodromy about $\partial D^2$. Cases (2) and (3) of \Cref{thm:symp-iso-intro} deal with those components where the monodromy is nontrivial and trivial respectively.

	Let $S \subseteq \mathcal{H}_n$ be one of the components specified by \Cref{thm:symp-iso-intro} and let $\widetilde{S}$ be its Galois closure over $\Conf_n(\operatorname{Int}(D^2))$ determined by the cover $\mathcal{H}_n \to \Conf_n(\operatorname{Int}(D^2))$. Given the notation above, \Cref{thm:symp-iso-intro} implies that $\pi_1(\widetilde{S})$ is equal to $\ker \rho$ or $\rho^{-1}(Z(U))$, where $\rho$ is the integral Burau representation reduced modulo three. 
\end{remark}

\para{Relationship to previous work}

The group $\Br(\pi)$ for genus $g$ Lefschetz fibrations $\pi : M \to S^2$ where $M$ is a symplectic 4-manifold has previously appeared in work of Auroux--Mu{\~n}oz--Presas \cite[Section 10]{lagrangian-and-lefschetz}. Using a construction of Donaldson, they show that any symplectomorphism of a symplectic 4-manifold $M'$ can be isotoped to a fiber-preserving diffeomorphism for some Lefschetz pencil on $M'$ (note that, after blowing up finitely many points, the pencil becomes a Lefschetz fibration). In the genus one setting, $\Br(\pi)$ was also studied by L\"{o}nne via the monodromy of families of elliptic fibrations \cite{lonne-braid-monodromy}. As previously mentioned, \Cref{thm:finite-classification} is closely related to work of Ito and L\"{o}nne on Hurwitz orbits for representations valued in the braid group (see \Cref{rmk:ito-lonne} above).

\vspace{0.2cm}
\para{Organization of the paper}

\Cref{sec:normal-core-triv} begins by recalling a condition for Zariski density developed by Coccia--Litt in \cite{coccia-litt}. We then prove \Cref{thm:z-dense} and \Cref{thm:normal-core-triv}, concluding our study of universal braids for smooth elliptic fibrations over the sphere. We address the finite-index examples in \Cref{sec:finite-stuff} and prove part \ref{item:virt-univ-cox} of \Cref{thm:finite-classification}. We then prove part \ref{item:virt-univ-class} of \Cref{thm:finite-classification} in \Cref{subsec:fin-class}. The key step is \Cref{lemma:all-edge}, which shows that if $\Br(\pi)$ is finite-index in $B_n$ then there exists a complete set of vanishing cycles consisting of two curves with intersection number one. To generalize the finite-index behavior, we carry out the aforementioned branched cover construction in \Cref{sec:branched-covers} and prove \Cref{thm:symp-iso-intro} in \Cref{subsec:sp-iso}. 

\vspace{0.2cm}
\para{Code/Data availability and AI Usage}

A program in Sage/Python was implemented to do the necessary computations throughout the paper. The program can be obtained from GitHub at
\begin{align*}
	\text{\href{https://github.com/FayeAlephNil/solomon}{https://github.com/FayeAlephNil/solomon}} \tag{$\dagger\dagger$}\label{eq:code}
\end{align*}
or, upon reasonable request, from the author. 

Gemini Pro was used to improve the figure in \Cref{fig:branched-from-lefschetz} from a hand-drawn picture into professional TikZ code. Gemini Pro was also used to perform literature searches for references before being hand-checked by the author. Claude Opus was used in the editing process to detect typos and check proofs.

\vspace{0.2cm}
\para{Acknowledgements}

The author would like to thank Benson Farb for his constant support and encouragement. The author thanks William Chen, Benson Farb, Daniel Litt, Eduard Looijenga, Aaron Landesman, Dan Margalit, Daniel Minahan, Ethan Pesikoff, and Zhong Zhang for many helpful conversations regarding the present work. In particular, to Benson Farb, Aaron Landesman, Dan Margalit, Ethan Pesikoff, and Zhong Zhang, the author is grateful for their helpful comments on earlier drafts of this paper.

\section{Proof of Zariski density and universal braids over the sphere}\label{sec:normal-core-triv}

The goal of this section is to prove \Cref{thm:z-dense} and to show it implies \Cref{thm:normal-core-triv}. Before doing so, we review a general condition for verifying Zariski density of the Hurwitz orbits of integral points on character varieties due to Coccia--Litt \cite{coccia-litt}. Let $\Sigma_{g,n}$ be a genus $g$ surface with $n$ punctures, and let $\Hom(\pi_1 \Sigma_{g,n}, \SL_2)$ denote the $\ZZ$-scheme whose $R$-points for a ring $R$ are given by
\begin{align*}
	\Hom(\pi_1 \Sigma_{g,n}, \SL_2(R)).
\end{align*}
We write the \textit{$\SL_2$-character variety} as the GIT quotient
\begin{align*}
	\mathfrak{X}(\Sigma_{g,n}) \coloneqq \Hom(\pi_1 \Sigma_{g,n},\SL_2)\sslash \SL_2
\end{align*}
whose $R$-points we denote by $\mathfrak{X}_R(\Sigma_{g,n})$.
\begin{remark}
	Note that, for the subset of $\Hom(\pi_1 \Sigma_{g,n},\SL_2(R))$ consisting of irreducible representations, the choice between the GIT, set-theoretic, and categorical quotients is arbitrary. In this section, all relevant representations are irreducible, and so from now on we will suppress this particular subtlety.
\end{remark}
Furthermore, let $\Mod(\Sigma_{g,n}) \coloneqq \pi_0(\Diff^+(\Sigma_{g,n}))$ denote the (impure) mapping class group of $\Sigma_{g,n}$ and $\PMod(\Sigma_{g,n}) < \Mod(\Sigma_{g,n})$ denote the subgroup of pure mapping classes, i.e., those which fix each puncture pointwise. As in the introduction, $\mathfrak{X}(\Sigma_{g,n})$ admits an algebraic action by the mapping class group $\Mod(\Sigma_{g,n})$. Furthermore the pure mapping class group $\PMod(\Sigma_{g,n})$ preserves the fibers of the map
\begin{align*}
	\tr_{\mathrm{punct}} : \mathfrak{X}_R(\Sigma_{g,n}) \to \mathbb{A}^n
\end{align*}
taking a representation $[\rho]$ to the traces $\tr(\rho(\gamma_i))$ where $\gamma_i$ is a simple closed curve homotopic to the $i$-th puncture. For $\vec{k} \in \mathbb{A}^n$ let
\begin{align*}
	\mathfrak{X}_{\vec{k},R}(\Sigma_{g,n}) \coloneqq \tr_{\mathrm{punct}}^{-1}(\vec{k})
\end{align*}
denote the relative character variety. When $\Sigma_{g,n}$ is clear from context, we will omit it from the notation. 
\begin{proposition}[Coccia--Litt, {\cite[Proposition 3.1.10]{coccia-litt}}]\label{prop:coccia-litt}
	Let $[\rho] \in \mathfrak{X}_{\vec{k},\overline{\Q}}(\Sigma_{g,n})$ and suppose there exists some pants decomposition $\mathcal{P} = \{a_1,\ldots,a_{3g-3+n}\}$ of $\Sigma_{g,n}$ such that
	\begin{enumerate}
		\item\label{cond:eig} The eigenvalues of $\rho(a_i)$ are not roots of unity, i.e., $\tr(\rho(a_i))$ does not lie in $E = \{e^{2\pi i \ell/j} + e^{-2\pi i\ell/j} \st \ell \in \ZZ, j \in \N\}$,
		\item\label{cond:irr} Either $(g,n,\vec{k}) = (1,1,2)$ or for each component $\Sigma$ of the surface $\Sigma_{g,n} - \mathcal{P}$ obtained by cutting along the pants curves $\{a_i\}$, the traces $x,y,z$ obtained from $\rho$ by evaluating along $\partial \Sigma$ do not satisfy
			\begin{align*}
				x^2 + y^2 + z^2 - xyz = 4,
			\end{align*}
	\end{enumerate}
	then $\PMod(\Sigma_{g,n}) \cdot [\rho]$ is Zariski dense in $\mathfrak{X}_{\vec{k},\overline{\Q}}(\Sigma_{g,n})$.
\end{proposition}
Condition \ref{cond:eig} can be rephrased as requiring that each $\rho(a_i)$ is neither parabolic nor of finite order. Likewise, the second half of condition \ref{cond:irr} of \Cref{prop:coccia-litt} is equivalent to requiring that $\rho$ is irreducible on each component $\Sigma$ of the cut surface $\Sigma_{g,n} - \mathcal{P}$ (see \cite[Lemma~3.3]{whang2020}). With \Cref{prop:coccia-litt} in hand, we prove \Cref{thm:z-dense}.
\begin{figure}[t!]
	\subfloat[A schematic tree depicting the pants decomposition on $S^2 \setminus \Delta$. Each leaf represents a puncture, and each non-leaf non-root node represents a cuff curve of the pants decomposition. The dashed edge indicates a continuation of the pattern above.]{
		\centering
		\makebox[0.9\textwidth]{
			\includestandalone[scale=0.8]{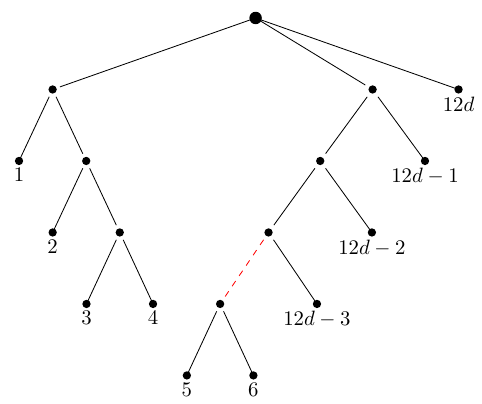}
		}
	}

	\subfloat[A graphical depiction of the pants decomposition on $S^2 \setminus \Delta$ in the plane, note that there is no puncture at $\infty$.]{
		\centering
		\makebox[0.9\textwidth]{
			\includegraphics{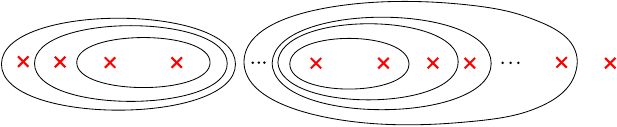}
		}
	}
	\caption{Two depictions of the pants decomposition on $S^2 \setminus \Delta$ used in the proof of \Cref{thm:z-dense}.}
	\label{fig:special-pants}

\end{figure}
\begin{proof}[Proof of \Cref{thm:z-dense}]
	Let $\pi : M \to S^2$ be a smooth elliptic fibration with $n = 12d$ nodal fibers along $\Delta \subseteq S^2$ and monodromy representation
	\begin{align*}
		\phi_\pi : \pi_1(S^2 \setminus \Delta) \to \SL_2\ZZ.
	\end{align*}
	We wish to show that $\Mod(S^2,\Delta) \cdot [\phi_\pi]$ is Zariski dense in the relative character variety $\mathfrak{X}_{\vec{k},\CC}(S^2 \setminus \Delta)$. Note that, a priori, $\phi_\pi$ is arbitrary. However, Moishezon's classification of smooth genus one fibrations (see \cite[Theorem 9 and Lemma 8]{moishezon}) implies that given \textit{any} sequence of matrices $A_1,\ldots,A_n \in \SL_2\ZZ$ such that $A_i$ is conjugate to $\begin{psmallmatrix} 1 & -1 \\ 0 & 1 \end{psmallmatrix}$ for each $i$ and
	\begin{align*}
		A_1 \cdots A_n = \Id,
	\end{align*}
	there exists a sequence of curves $\gamma_1,\ldots,\gamma_n$ surrounding each puncture of $S^2 \setminus \Delta$ so that $\phi_\pi(\gamma_i) = A_i$. Our goal is to find such a sequence of matrices $A_1,\ldots,A_n$ which is amenable to applying \Cref{prop:coccia-litt}, see \Cref{rem:finding-matrices} below for an explanation of the heuristics we used to find the following sequence. 

	\step{Constructing the factorization and the pants decomposition}
	\noindent Write the matrices
	\begin{align*}
		A &= \begin{pmatrix} 1 & -1 \\ 0 & 1 \end{pmatrix} && B = \begin{pmatrix} 1 & 0 \\ 1 & 1 \end{pmatrix} \\
		C &= 
		\begin{pmatrix} 
			-6 & -1 \\
			49 & 8
		\end{pmatrix} && D = 
		\begin{pmatrix} 
			21 & -16 \\
			25 & -19
		\end{pmatrix} \\
		E &= \begin{pmatrix} 0 & -1 \\ 1 & 2 \end{pmatrix},
	\end{align*}
	and note that each of $A,B,C,D,E$ is conjugate to the standard shear $A$. Furthermore, note that
	\begin{equation}
		(BBCADB A A A B A B)(EA(AB)^2 EA(AB)^2)^{d-1} = \Id \cdot (\Id)^{d-1} = \Id,\label{eq:special-mon}
	\end{equation}
	and so by Moishezon's theorem there are some simple closed curves $\gamma_1,\ldots,\gamma_{12d}$ in $S^2 \setminus \Delta$ so that $\gamma_i$ is a curve about puncture $i$ and
	\begin{align*}
		\phi_\pi(\gamma_1) \cdots \phi_\pi(\gamma_{12d}) = \Id
	\end{align*}
	is the monodromy factorization given in \eqref{eq:special-mon}. Let $\mathcal{P}$ denote the pants decomposition determined by $\gamma_1,\ldots,\gamma_{12d}$ and the tree depicted in \Cref{fig:special-pants}. Each leaf node in the tree corresponds to a puncture $i$ so that travelling the ``spine'' of the pants decomposition and looping about puncture $i$ gives the curve $\gamma_i$. Each node which is neither a leaf nor the root node corresponds to a cuff in the pants decomposition.

	\step{Verifying condition \ref{cond:eig} of \Cref{prop:coccia-litt}}
	\noindent Denote the cuff curves in the left branch of the tree by $a_1,a_2,a_3$ and the cuff curves in the middle branch of the tree by $a_4,\ldots,a_{n-3}$, where $n = 12d$ is the number of punctures. Furthermore, let $\mathcal{P} = \{a_1,\ldots,a_{n-3}\}$ be the induced pants decomposition. Below we abbreviate $\tr(\phi_\pi(a_i))$ by $\tr(a_i)$. Direct computation shows that
	\begin{align}
		\tr(a_1) = -47 && \tr(a_2) = -42 && \tr(a_3) = -37.\label{eq:trace-small}
	\end{align}
	Furthermore for $j \geq 10$
	\begin{align*}
		\phi_\pi(a_j) = \pm \phi_\pi(a_{j + 6}),
	\end{align*}
	because $EA(AB)^2 = -\Id$. Thus, one may compute that
	\begin{align}
		\tr(a_4) &= -14 && \tr(a_5) = -20 && \tr(a_6) = -26 \nonumber \\
		\tr(a_7) &= -32 && \tr(a_8) = -63 && \tr(a_9) = -32 \nonumber \\
		\tr(a_{10}) &= -37 && \tr(a_{11}) = 20 && \tr(a_{12}) = 26 \nonumber \\
		\tr(a_{13}) &= 32 && \tr(a_{14}) = 63 && \tr(a_{15}) = 32 \label{eq:trace-big}
	\end{align}
	(see the Sage implementation at \eqref{eq:code}). Thus condition \ref{cond:eig} of \Cref{prop:coccia-litt} holds since $\abs{\tr(a_i)} > 2$ for all $i$. 

	\step{Verifying condition \ref{cond:irr} of \Cref{prop:coccia-litt}}
	\noindent Note that each valence three node in the tree given by \Cref{fig:special-pants} represents a component of the surface $(S^2 \setminus \Delta) - \mathcal{P}$ obtained by cutting along the cuff curves $a_i$. Furthermore, each component contains at least one puncture. Let $\Sigma$ be one component of $(S^2 \setminus \Delta) - \mathcal{P}$; we must verify that the boundary traces $x,y,z$ satisfy
	\begin{align}
		x^2 + y^2 + z^2 - xyz \neq 4.\label{eq:neq-4}
	\end{align}
	At least one of $x,y,z$ is equal to two, since $\Sigma$ contains a puncture and $\tr(\phi_\pi(\gamma_i)) = 2$ for all $\gamma_i$. Thus, \eqref{eq:neq-4} holds if and only if $x \neq y$, where $x,y$ are the remaining boundary curves. Direct analysis of the pair of pants given by \Cref{fig:special-pants} and the traces computed in \eqref{eq:trace-small} and \eqref{eq:trace-big} shows that $x \neq y$ for all pairs of pants in the given decomposition. Therefore condition \ref{cond:irr} of \Cref{prop:coccia-litt} holds. Applying \Cref{prop:coccia-litt} shows that $\PMod(S^2,\Delta) \cdot [\phi_\pi]$ is Zariski dense in the relative character variety $\mathfrak{X}_{\vec{k},\overline{\Q}}(S^2 \setminus \Delta)$ where $\vec{k} = (2,2,\ldots,2)$ as desired. Base change then implies Zariski density in $\mathfrak{X}_{\vec{k},\CC}(S^2 \setminus \Delta)$.
\end{proof}

\begin{remark}\label{rem:finding-matrices}
	The arbitrary nature of the monodromy factorization produced in \eqref{eq:special-mon} is mildly misleading. Let $A,B,C,D,E \in \SL_2\ZZ$ be represented by Dehn twists about simple closed curves $\alpha,\beta,\gamma,\de,\epsilon$ in the torus respectively. The general method used to produce the factorization was as follows. 
	\begin{enumerate}
		\item Start with the ``standard'' monodromy factorization
			\begin{align*}
				(T_\alpha T_\beta)^{6d} = \Id
			\end{align*}
			and apply Hurwitz moves to create two curves $\gamma,\de$ with high intersection number with each other and with $\alpha,\beta$.
		\item Use the pants decomposition to cordon off the vanishing cycles $\gamma,\de$ so that the traces of boundary curves are large.
		\item Apply a small number of Hurwitz moves to the remaining $12(d-1)$ matrices to prevent the product from ever having trace two.
	\end{enumerate}
	Each portion of this plan was carried out using a combination of computer search along with the blood, sweat, and tears of the author.
\end{remark}

To end the section, we formally complete the proof that \Cref{thm:z-dense} implies \Cref{thm:normal-core-triv}. To do so requires a lemma embedding Teichm\"{u}ller space in the $\SL_2$-character variety. We believe \Cref{lemma:embed} will be familiar to experts, and the compact case can be found in Goldman's work \cite{goldman-geom-struct}. For the convenience of the reader, we sketch the argument of the genus zero punctured case here.
\begin{lemma}\label{lemma:embed}
	Let $\Sigma_{0,m}$ denote the genus $0$ surface with $m$ punctures, then the Teichm\"{u}ller space $\mathcal{T}_{0,m}$ of marked hyperbolic structures on $\Sigma_{0,m}$ embeds into $\mathfrak{X}_{\vec{k}',\CC}(\Sigma_{0,m})$ for $\vec{k}' = (-2,\ldots,-2)$. Furthermore, the embedding is $\PMod(\Sigma_{0,m})$-equivariant.
\end{lemma}

\begin{proof}
	Recall that for $\abs{\Delta} = m$ a finite subset of $S^2$,
	\begin{align*}
		\mathcal{T}_{0,m} = \operatorname{DF}_{\mathrm{para}}(\pi_1(S^2 \setminus \Delta), \PSL_2\R)/\PGL_2\R,
	\end{align*}
	where $\operatorname{DF}_{\mathrm{para}}$ denotes the space of discrete faithful representations such that each simple closed curve $\gamma_i$ about a puncture $p_i \in \Delta$ is sent to a parabolic element of $\PSL_2\R$ and the action of $\PGL_2\R$ is by conjugation. Each element $\rho$ of $\operatorname{DF}_{\mathrm{para}}(\pi_1(S^2 \setminus \Delta),\PSL_2\R)$ admits a uniquely specified lift $\widetilde{\rho}$ to $\operatorname{DF}_{\mathrm{para}}(\pi_1(S^2 \setminus \Delta),\SL_2\R)$ by choosing $\tr(\widetilde{\rho}(\gamma_i)) = -2$ for $i = 1,\ldots,m-1$. The sign of $\tr(\widetilde{\rho}(\gamma_m))$ is then determined by the equation
	\begin{align*}
		\widetilde{\rho}(\gamma_1) \cdots \widetilde{\rho}(\gamma_m) = \Id.
	\end{align*}
	Because an element of $\mathcal{T}_{0,m}$ is determined by the lengths of $3m-9$ curves (see \cite[Theorem 10.7]{primer} for the compact case, which is analogous), the induced map $\iota : \mathcal{T}_{0,m} \to \mathfrak{X}_\CC(S^2 \setminus \Delta)$ is injective. Continuity of $\iota$ similarly follows because the map to $\mathfrak{X}_\CC(S^2\setminus\Delta)$ is determined by the traces along simple closed curves, and these traces are determined by the length functions on $\mathcal{T}_{0,m}$. Therefore, $[\rho] \mapsto \operatorname{sign}(\tr(\widetilde{\rho}(\gamma_m)))$ is a continuous function on $\mathcal{T}_{0,m}$ taking values in $\pm 1$. Since $\mathcal{T}_{0,m}$ is connected, $\operatorname{sign}(\tr(\widetilde{\rho}(\gamma_m)))$ is constant over $\mathcal{T}_{0,m}$. Thus, to show that $\iota$ maps $\mathcal{T}_{0,m}$ into $\mathfrak{X}_{\vec{k}',\CC}(S^2 \setminus \Delta)$ it suffices to construct a single hyperbolic structure on $\Sigma_{0,m}$ admitting a lift to $\SL_2\R$ where all the peripheral traces are $-2$. We do so below in \Cref{lemma:one-hyper-lift}. Finally, to show $\PMod(S^2, \Delta)$-equivariance, note first that $\Mod(S^2,\Delta)$ in fact acts on both sides. Given $\sigma_i$ a generator of $\Mod(S^2,\Delta)$ it suffices to show that $\widetilde{\rho} \circ (\sigma_i)_\ast$ is a lift of $\rho \circ (\sigma_i)_\ast$ with peripheral traces all $-2$. Because $(\sigma_i)_\ast : F_m \to F_m$ exchanges and conjugates the generators $\gamma_1,\ldots,\gamma_m$, the equivariance follows from the fact that conjugation preserves traces.
\end{proof}

\begin{lemma}\label{lemma:one-hyper-lift}
	Let $X$ denote the hyperbolic structure on $\Sigma_{0,m}$ given by the $(m-2)$-sheeted branched cover of $\PP^1 \setminus \{0,1,\infty\}$ completely ramified at $0$ and $\infty$. Then the representation $\rho_X : \pi_1(\Sigma_{0,m}) \to \PSL_2\R$ corresponding to the hyperbolic structure $X$ admits a lift where all the peripheral traces are $-2$.
\end{lemma}
\begin{proof}
	Recall that $Y = \PP^1 \setminus \{0,1,\infty\}$ admits a unique complex (and thereby hyperbolic) structure since $\PSL_2 \CC$ acts 3-transitively on $\PP^1$. One incarnation of this structure is the modular curve $\mathbb{H}^2/\overline{\Gamma}(2)$, where $\overline{\Gamma}(2) \coloneqq \ker(\PSL_2\ZZ \to \PSL_2 \mathbb{F}_2)$ is the level 2 congruence subgroup of $\PSL_2\ZZ$.

	From this description, we see that the monodromy representation corresponding to $Y$ has a lift to $\SL_2\R$ given by
	\begin{align*}
		\widetilde{\rho}_Y \coloneqq \pi_1(Y) &\to \SL_2\R \\
		\widetilde{\rho}_Y(\gamma_0) &= \begin{pmatrix} -1 & 2 \\ 0& -1\end{pmatrix}  \\
		\widetilde{\rho}_Y(\gamma_1) &= \begin{pmatrix} -1 & 0 \\ -2 & -1\end{pmatrix} \\
		\widetilde{\rho}_Y(\gamma_\infty) &= \begin{pmatrix} 1 & 2 \\-2 & -3 \end{pmatrix}  
	\end{align*}
	where $\gamma_0,\gamma_1,\gamma_\infty$ are peripheral generators of $\pi_1(Y)$ so that $\gamma_0\gamma_1\gamma_\infty = 1$. The covering $X \to Y$ is given by the subgroup generated by 
	\begin{align*}
		\de_0 \coloneqq \gamma_0^{m-2} && \de_{\zeta^i} \coloneqq \gamma_0^{i}\gamma_1\gamma_0^{-i} && \de_\infty \coloneqq \gamma_\infty^{m-2},
	\end{align*}
where $\zeta$ is a primitive $(m-2)$-th root of unity and $i$ ranges from $0$ to $m-3$. In this way, $\pi_1(X)$ is presented as 
	\[
		\left\langle \de_0,\de_{\zeta^i}, \de_\infty \mid  \de_0\left(\prod_{i=0}^{m-3} \de_{\zeta^i}\right)\de_\infty = 1\right\rangle
	\]
	Choose a lift $\widetilde{\rho}_X$ of $\rho_X$ to $\SL_2\R$ so that the trace of $\widetilde{\rho}_X(\de_0)$ is $-2$ and the trace of $\widetilde{\rho}_X(\de_{\zeta^i})$ is $-2$ for all $i$. One may then compute that
	\begin{align*}
		\widetilde{\rho}_X(\de_0) &= (-1)^{m-1}\widetilde{\rho}_Y(\gamma_0)^{m-2} \\
		\widetilde{\rho}_X(\de_{\zeta^i}) &= \widetilde{\rho}_Y(\gamma_0)^{i} \widetilde{\rho}_Y(\gamma_1) \widetilde{\rho}_Y(\gamma_0)^{-i},
	\end{align*}
	since trace is conjugation invariant. Thus
	\begin{align*}
		\widetilde{\rho}_X(\de_\infty)^{-1} &= (-1)^{m-1}\widetilde{\rho}_Y(\gamma_0)^{m-2} \prod_{i=0}^{m-3} \widetilde{\rho}_Y(\gamma_0)^{i}\widetilde{\rho}_Y(\gamma_1) \widetilde{\rho}_Y(\gamma_0)^{-i} \\
									&= (-1)^{m-1}\widetilde{\rho}_Y(\gamma_0)^{m-3} (\widetilde{\rho}_Y(\gamma_0)\widetilde{\rho}_Y(\gamma_1))^{m-2} \widetilde{\rho}_Y(\gamma_0)^{3-m} \\
									&= (-1)^{m-1}\widetilde{\rho}_Y(\gamma_0)^{m-3} (\widetilde{\rho}_Y(\gamma_{\infty})^{-1})^{m-2} \widetilde{\rho}_Y(\gamma_0)^{3-m}
	\end{align*}
	Now since $\widetilde{\rho}_Y(\gamma_\infty)^{-1}$ has eigenvalues $-1,-1$ we conclude that
	\begin{align*}
		\tr(\widetilde{\rho}_X(\de_\infty)) = (-1)^{m-1}(-1)^{m-2} \cdot 2  = -2
	\end{align*}
	as desired.
\end{proof}

With \Cref{lemma:embed,lemma:one-hyper-lift} in place, we prove \Cref{thm:normal-core-triv}.
\begin{proof}[Proof of \Cref{thm:normal-core-triv}]
	Let $\pi : M \to S^2$ be an elliptic fibration with $n$ nodal fibers along $\Delta \subseteq S^2$. We wish to show that $\Core(\Br(\pi)) < \Mod(S^2,\Delta)$ is trivial. We first show that
	\begin{align*}
		\Gamma \coloneqq \Core(\Br(\pi)) \cap \PMod(S^2, \Delta)
	\end{align*}
	is trivial. By \Cref{thm:z-dense}, the Hurwitz orbit of the monodromy representation $\phi_\pi$ is Zariski dense in the relative character variety $\mathfrak{X}_{\vec{k},\CC}(S^2 \setminus \Delta)$ for $\vec{k} = (2,2,2,\ldots,2)$. Therefore, we have that $\Core(\Br(\pi))$ acts by the identity on $\mathfrak{X}_{\vec{k},\CC}(S^2 \setminus \Delta)$, since it fixes the entire Hurwitz orbit of $\phi_\pi$ and the action of $\Mod(S^2,\Delta)$ is algebraic. Thus, to show $\Gamma = 1$ it suffices to show that $\PMod(S^2,\Delta)$ acts faithfully on $\mathfrak{X}_{\vec{k},\CC}(S^2 \setminus \Delta)$ where $\vec{k} = (2,\ldots,2)$. By the classification of elliptic fibrations, $n = \abs{\Delta} = 12d$ for some $d \geq 1$ \cite[Theorem 9]{moishezon}, and hence $\mathfrak{X}_{\vec{k},\CC}(S^2 \setminus \Delta)$ is equivariantly isomorphic to $\mathfrak{X}_{\vec{k}',\CC}(S^2 \setminus \Delta)$ where $\vec{k}' = (-2,\ldots,-2)$. The isomorphism is given by taking $[\rho]$ to $\rho'$ where $\rho'(\gamma_i) = -\rho(\gamma_i)$ for a standard set of generators $\gamma_1,\ldots,\gamma_n$ of $\pi_1(S^2 \setminus \Delta)$. Therefore, by \Cref{lemma:embed}, $\mathcal{T}_{0,n}$ equivariantly embeds into $\mathfrak{X}_{\vec{k}',\CC}(S^2 \setminus \Delta) \cong \mathfrak{X}_{\vec{k},\CC}(S^2 \setminus \Delta)$. Because $\PMod(S^2, \Delta)$ is centerless, it acts faithfully on $\mathcal{T}_{0,n}$ \cite[Sections 3.4 and 12.1]{primer}. We conclude that $\PMod(S^2,\Delta)$ acts faithfully on $\mathfrak{X}_{\vec{k},\CC}(S^2 \setminus \Delta)$ as well.

	Since $\Core(\Br(\pi)) \cap \PMod(S^2,\Delta) = 1$, it follows that $\Core(\Br(\pi))$ injects into $S_n$ and so $\Core(\Br(\pi))$ is finite. However, standard arguments in mapping class group theory imply that $\Mod(S^2,\Delta)$ contains no finite normal subgroups when $\abs{\Delta} \geq 5$ (see \cite[Sections 11.5 and 11.6]{ivanov-teich-groups}). By Moishezon's theorem, any nontrivial smooth elliptic fibration with only nodal fibers has at least 12 singular fibers, concluding the proof \cite[Theorem 9]{moishezon}.
\end{proof}
In light of \Cref{thm:normal-core-triv} we make the following conjecture.
\begin{conjecture}
	If $\pi : M \to S^2$ is any genus $g$ Lefschetz fibration with simple singularities whose monodromy representation surjects onto $\Mod(\Sigma_g)$, then $\Core(\Br(\pi))$ is trivial.
\end{conjecture}

\section{Universal braids over the disk}\label{sec:finite-stuff}

\subsection{The virtually universal examples}\label{subsec:finite}

The aim of this section is to analyze the finite-index examples classified by part \ref{item:virt-univ-class} of \Cref{thm:finite-classification} and to prove part \ref{item:virt-univ-cox} of \Cref{thm:finite-classification}. The following examples were the impetus for the current project and for the classification, and so deserve special care and attention. Before we begin, we set some notation. Throughout, let $\gamma_1,\ldots,\gamma_n \in \pi_1(D^2 \setminus n \text{ points})$ be a standard system of generators so that each $\gamma_i$ surrounds one puncture and $\gamma_1\cdots \gamma_n$ is homotopic to $\partial D^2$. Furthermore, let $q_n : M_n \to D^2$ be the genus one Lefschetz fibration over $D^2$ defined by the monodromy representation
\begin{align*}
	\phi_n : \pi_1(D^2 \setminus n \text{ points}) &\to \SL_2\ZZ \\
	\gamma_{2i+1} &\mapsto \begin{pmatrix} 1 & -1 \\ 0 & 1 \end{pmatrix}  \\
	\gamma_{2i} &\mapsto \begin{pmatrix} 1 & 0 \\ 1 & 1 \end{pmatrix}.
\end{align*}
Part \ref{item:virt-univ-class} of \Cref{thm:finite-classification}, to be proved in the next section, states that for any genus one Lefschetz fibration $\pi : M \to D^2$ such that $[B_n : \Br(\pi)] < \infty$ either:
\begin{enumerate}
	\item The monodromy group of $\pi$ is abelian, and $\phi_\pi(\gamma_i) = \phi_\pi(\gamma_j)$ for all $i,j$.
	\item The number of singular fibers $n$ is at most $4$ and there exist diffeomorphisms $F : M \to M_n$ and $f : D^2 \to D^2$ such that $\rest{F}{\pi^{-1}(\partial D^2)} = \Id$ so that
		\begin{center}
			\begin{tikzcd}
				M \ar[d,"\pi"'] \ar[r,"F"] & M_n \ar[d,"q_n"]\\
				D^2 \ar[r,"f"] & D^2
			\end{tikzcd}
		\end{center}
		commutes. Equivalently, due to a result of Moishezon, $n \leq 4$ and $\phi_\pi$ lies in the same Hurwitz orbit as $\phi_n$ (see \cite[Theorem~2.1]{how-large}).
\end{enumerate}
We now focus on the examples $q_2,q_3,$ and $q_4$. The Hurwitz orbits of $\phi_3$ and $\phi_4$, computed in Sage/Python, are displayed in \Cref{fig:braid-orbits}. The orbit for $\phi_2$ is simply a triangle.
\begin{figure}
	\centering
	\subfloat[$n=3$]{
		\centering
		\includegraphics[scale=0.45]{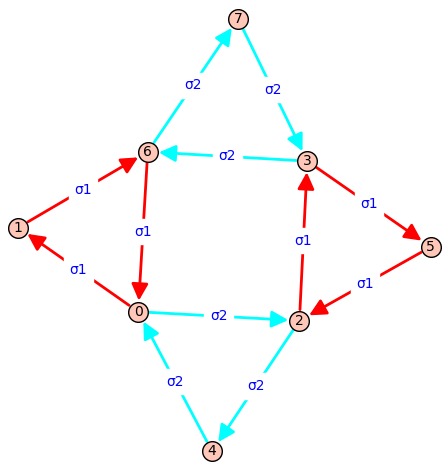}
	}
	\hspace{0.1in}
	\subfloat[$n=4$]{
		\centering
		\includegraphics[scale=0.3]{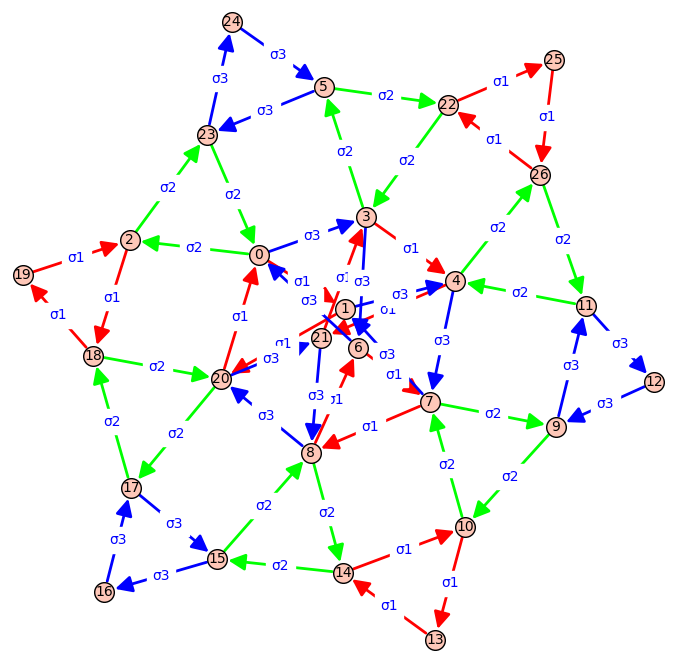}
	}\\
	\subfloat[Center of $n=4$]{
		\centering
		\includegraphics[scale=0.3]{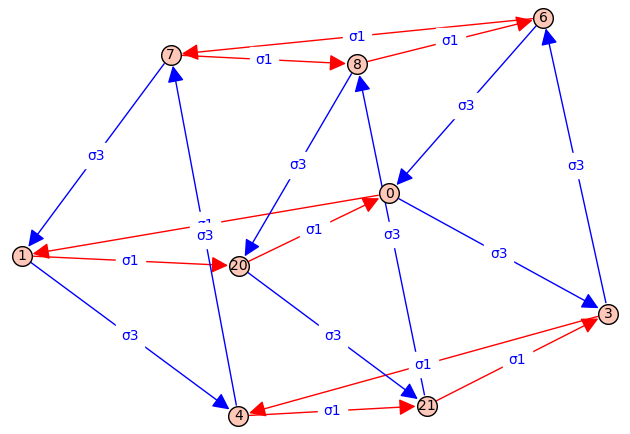}
	}
	\caption{Hurwitz orbits of $\phi_n$ for $n = 3$ and $4$ under the action of $B_n$. Vertex $0$ is $\phi_{n}$, and missing edges indicate representations fixed by the corresponding $\sigma_i$. When $n=3$ the orbit has size $8$, and when $n = 4$ it has size $27$. Figure (C) is the center portion of Figure (B).}
	\label{fig:braid-orbits}
\end{figure}
An examination of \Cref{fig:braid-orbits} reveals that $\sigma_1^3$ fixes each point in the Hurwitz orbit, and thus $\langle \langle \sigma_1^3 \rangle  \rangle < \Core(\Br(q_n))$. In fact, we have the following, which forms part \ref{item:virt-univ-cox} of \Cref{thm:finite-classification}.
\begin{proposition}\label{prop:core-cubes}
	Let $n \in \{2,3,4\}$, then
	\begin{align*}
		\Core(\Br(q_n)) = \langle \langle \sigma_1^3 \rangle  \rangle,
	\end{align*}
	where $\langle \langle \sigma_1^3 \rangle  \rangle$ denotes the normal closure of a single half-twist cubed.
\end{proposition}
One can prove \Cref{prop:core-cubes} via a direct computation in GAP by giving generators for $\Br(q_n)$ when $n \leq 4$ as in \cite[Proposition 4.2]{how-large}. The groups $\langle \langle \sigma_1^3 \rangle  \rangle$ appear also in the following theorem of Coxeter.
\begin{theorem}[Coxeter, {\cite{coxeter-factor-braid}}]\label{thm:coxeter}
	Let $n,k \geq 3$. The factor group $B_n(k) \coloneqq B_n/\langle \langle \sigma_1^k \rangle  \rangle$ is finite if and only if there is a regular polyhedron consisting of regular $n$-gons arranged with $k$ faces meeting at a vertex. In particular, the only finite cases are the Schl\"{a}fli symbols $\{n,k\} = \{3,3\}, \{3,4\}, \{3,5\}, \{4,3\}, \{5,3\}$ of the platonic solids. Furthermore, in these cases
	\begin{align*}
		\abs{B_n(k)} = \left(\frac{f_{n,k}}{2}\right)^{n-1}n!
	\end{align*}
	where $f_{n,k}$ is the number of faces in the aforementioned regular polyhedron.
\end{theorem}
Equipped with \Cref{thm:coxeter} we give an alternate proof of \Cref{prop:core-cubes}, then we discuss the relationship of \Cref{prop:core-cubes} and \Cref{thm:coxeter} to the platonic solids (see \Cref{rem:coxeter-platonic}). When $n = 2$, we know that $B_2 \cong \ZZ$ is abelian and so $\Core(\Br(q_2)) = \Br(q_2) = \langle \langle \sigma_1^3 \rangle  \rangle$.
\begin{proof}[Proof of \Cref{prop:core-cubes} when $n = 3$]
	First note that $\abs{B_3(3)} = 2^2 \cdot 3! = 24$ by Coxeter's formula in \Cref{thm:coxeter}, and so
	\begin{align*}
		24 = [B_3 : \langle \langle \sigma_1^3 \rangle  \rangle] = [B_3 : \Br({q_3})][\Br({q_3}) : \Core(\Br({q_3}))][\Core(\Br({q_3})) : \langle \langle \sigma_1^3 \rangle  \rangle ].
	\end{align*}
	Since $\Br({q_3}) \neq B_3$ and it contains $\sigma_1^{-1}\sigma_2\sigma_1$, it cannot be normal, hence $[\Br({q_3}) : \Core(\Br({q_3}))] \neq 1$. Then, because $[B_3 : \Br({q_3})] = 8$, this implies that $[\Br({q_3}) : \Core(\Br({q_3}))] = 3$ and $\Core(\Br({q_3})) = \langle \langle \sigma_1^3 \rangle  \rangle$.
\end{proof}

\begin{proof}[Proof of \Cref{prop:core-cubes} when $n = 4$]
	We find that $\abs{B_4(3)} = 3^3 \cdot 4! = 27 \cdot 24$ by Coxeter's formula, and so, as before
	\begin{align*}
		27 \cdot 24 = [B_4 : \langle \langle \sigma_1^3 \rangle  \rangle] = [B_4 : \Br({q_4})][\Br({q_4}) : \Core(\Br({q_4}))][\Core(\Br({q_4})) : \langle \langle \sigma_1^3 \rangle  \rangle ].
	\end{align*}
	However, unlike the previous case when $n = 3$ one must explicitly compute the index $[\Br({q_4}) : \Core(\Br({q_4}))]$. It suffices to show the index is $24$ as $[B_4 : \Br({q_4})] = 27$. To do so, notice that $\Core(\Br({q_4}))$ is the kernel of the map $B_4 \to \Perm(\mathcal{O}_4)$, where $\mathcal{O}_4$ is the Hurwitz orbit of $\phi_4$ and $\Perm(\mathcal{O}_4)$ is the symmetric group on the set $\mathcal{O}_4$. Put another way, if $\mathcal{O}_4 = \{\rho_{4,1},\ldots,\rho_{4,27}\}$, then $\Core(\Br({q_4}))$ is the stabilizer of the point $(\rho_{4,1},\ldots,\rho_{4,27})$ under the action of $B_4$ on $\mathcal{O}_4^{27}$ componentwise. Hence, it suffices to compute the size of the orbit of $(\rho_{4,1},\ldots,\rho_{4,27})$ under $\Br({q_4})$. The computation is done in Sage/Python as before and verifies that the orbit has size 24. Hence $[\Br(q_4) : \Core(\Br(q_4))] = 24$ and $\Core(\Br(q_4)) = \langle \langle \sigma_1^3 \rangle  \rangle$ as desired.
\end{proof}

\begin{remark}\label{rem:coxeter-platonic}
	\Cref{prop:core-cubes} and \Cref{thm:coxeter} suggest a modular relationship between elliptic fibrations and the platonic solids. However, that observation only scratches their surface. For example, $B_4/\langle \langle \sigma_1^3 \rangle  \rangle$ is identified with the action of the Hessian group on the 27 lines of the Hesse cubic pencil. Similarly, note that the Hurwitz orbit of $\phi_4$ consists of 27 distinct points. In fact, the image of $\Br(q_4)$ under the map $B_4 \to B_4/\Core(\Br(q_4))$ is the precise subgroup of index 27 inducing the action on the 27 lines which is identified by Coxeter \cite[Section 4]{coxeter-factor-braid}. Thus far, a modular interpretation of this relationship between genus one fibrations, Coxeter's factor groups, and platonic solids has eluded the author. For more on this phenomenon, see Lam--Landesman--Litt's paper on finite braid group orbits on $\SL_2$-character varieties and complex reflection groups \cite{lam-finite}.
\end{remark}
A notable feature of \Cref{prop:core-cubes} and the classification given in \Cref{thm:finite-classification} is the absence of the case when $n = 5$. In Coxeter's factor groups, the case $n = 5$ corresponds to the regular dodecahedron. We attempt to remedy the absence of the dodecahedron in \Cref{sec:branched-covers} by reducing modulo 2, noting that for $n = 2,3,4$ the structure of the Hurwitz orbits of $q_n$ remains unchanged after reduction modulo 2. Our remedy nearly succeeds in reproducing $B_5(3)$, but unfortunately it gives rise to a whole infinite family of finite-index subgroups in $B_n$. In fact, a direct computation with the group $\Core(\Br(q_5[2]))$ in \Cref{thm:symp-iso-intro} analogous to those given above shows it contains $\langle \langle \sigma_1^3 \rangle  \rangle$ with index $3$.

\subsection{Classification of fibrations with virtually universal braid group}\label{subsec:fin-class}

In this section, we prove part \ref{item:virt-univ-class} of \Cref{thm:finite-classification}. The techniques are similar to those presented in \cite[Section 5]{how-large}, and we adopt the notation of a Moishezon spider which we introduced there.
\begin{defn}[\textbf{Moishezon spider}, see {\cite[Definition 5.1]{how-large}}]\label{defn:spider}
	Let $\pi : M \to B$ be a Lefschetz fibration of genus $g$ over a compact connected oriented base $B$ with singular locus $\Delta_\pi$. A \emph{Moishezon spider} for $\pi$ consists of the following data
	\begin{enumerate}
		\item An embedded disk $D$ in $B$, so that $\partial D \cap \Delta_\pi = \0$.
		\item A basepoint $b \in \partial D$, where if $B$ has boundary we require $b \in \partial D \cap \partial B$.
		\item A collection of arcs $a_1,\ldots,a_k$ in $D$ beginning at $b$ and ending at distinct singular values $p_1,\ldots,p_k \in \operatorname{int}(D) \cap \Delta_\pi$ of $\pi$, disjoint from $\partial D$ except at $b$. The arcs $a_1,\ldots,a_k$ are ordered counterclockwise by how they meet a sufficiently small neighborhood of $b$.
	\end{enumerate}
	The \textit{topological type} of a Moishezon spider is the collection of vanishing cycles $(\de_1,\ldots,\de_k)$ associated to $a_1,\ldots,a_k$ via $\pi$ in the fiber $\Sigma_g \cong \pi^{-1}(b)$ and is well-defined up to the action of $\Mod(\Sigma_g)$. The arcs $a_1,\ldots,a_k$ are called the \textit{legs of the spider}. A subset $I \subseteq \{1,\ldots,k\}$ defines a \textit{subspider} $(a_i)_{i \in I}$.
\end{defn}
Let $\pi : M \to D^2$ be an arbitrary genus one Lefschetz fibration with $n$ nodal fibers, and denote its monodromy representation by $\phi_\pi$. Note first that $\Br(\pi) = \Stab \phi_\pi$ and so the index is the size of the Hurwitz orbit by Moishezon's theorem (see \cite[Theorem~2.1]{how-large}). Furthermore, for any embedded disk $D' \subseteq D^2$ which shares the basepoint $b \in \partial D^2 \cap \partial D'$, we have that the Hurwitz orbit for $\rest{\pi}{\pi^{-1}(D')}$ is contained in the Hurwitz orbit for $\pi$. Thus, when proving the index is infinite, it suffices to do so on an embedded disk. Likewise, the Hurwitz action can be computed in terms of Moishezon spiders. Let $(\ell_1,\ell_2)$ be a Moishezon spider with vanishing cycles $(\nu_1,\nu_2)$. On the tubular neighborhood of $\ell_1 \cup \ell_2$ the half-twist $\sigma$ acts on the representation by
\begin{align*}
	(T_{\nu_1},T_{\nu_2}) \xmapsto{\sigma} (T_{\nu_1}T_{\nu_2}T_{\nu_1}^{-1}, T_{\nu_1}).
\end{align*}
On the level of vanishing cycles, $\sigma$ thus takes
\begin{align*}
	(\nu_1,\nu_2) \mapsto (T_{\nu_1} \nu_2, \nu_1).
\end{align*}

In light of these observations, the proof of part \ref{item:virt-univ-class} of \Cref{thm:finite-classification} begins with the following lemma.
\begin{lemma}[{\cite[Lemma 5.1]{how-large}}]\label{lemma:int2-inf}
	Let $\pi$ be a genus one Lefschetz fibration over $D^2$ and suppose $\pi$ admits a Moishezon spider of topological type $(\gamma,\de)$, where $i(\gamma,\de) \geq 2$. Then the Hurwitz orbit of $\phi_\pi$ is infinite and so $[B_n : \Br(\pi)] = \infty$.
\end{lemma}
Suppose then that $\Br(\pi)$ were finite-index. As a consequence, given \textit{any} complete Moishezon spider $(\ell_1,\ldots,\ell_n)$ for $\pi$, the vanishing cycles $(\nu_1,\ldots,\nu_n)$ must all lie on some \textit{Farey triangle} (i.e., a triangle in the Farey complex, see \Cref{fig:farey}). In fact, we show that there \textit{exists} a complete Moishezon spider so that all the vanishing cycles lie on a Farey edge in \Cref{lemma:all-edge} below.

Let $\alpha,\beta,\gamma$ be a Farey triangle (with vertices ordered by the counterclockwise orientation). Let $\alpha,\beta,\de$ be the opposite triangle (which shares the edge along $\alpha,\beta$). Concretely, write
\begin{align*}
	\alpha = \pm\begin{pmatrix} 1 \\ 0 \end{pmatrix} && \beta = \pm\begin{pmatrix} 0 \\ 1 \end{pmatrix} && \gamma = \pm\begin{pmatrix} 1 \\ 1 \end{pmatrix} && \de = \pm\begin{pmatrix} 1 \\ -1 \end{pmatrix}.
\end{align*}
Note that $T_\alpha \beta = \de$ and $T_\beta \alpha = \gamma$; this relation is what orients triples of curves clockwise or counterclockwise.

\begin{figure}
	\centering
	\includestandalone[scale=1.2]{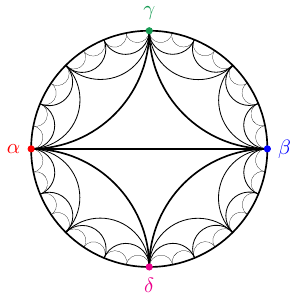}
	\caption{The Farey complex with a counterclockwise triangle $(\alpha,\beta,\gamma)$ and a clockwise triangle $(\alpha,\beta,\de)$. Recall that the Farey complex's vertices are simple closed curves in $T^2$ with edges between any two simple closed curves with geometric intersection number $1$.
	}
	\label{fig:farey}
\end{figure}

\begin{lemma}\label{lemma:clockwise}
	Let $\pi$ be a genus one Lefschetz fibration over $D^2$ and suppose $\pi$ admits a Moishezon spider of topological type $(\alpha,\beta,\gamma)$. Then $[B_n : \Br(\pi)] = \infty$.
\end{lemma}

\begin{proof}
	Suppose the Moishezon spider from the assumption is given by arcs $\ell_1,\ell_2,\ell_3$ beginning at $b \in \partial D^2$. We then perform a local calculation. Let $\sigma_1,\sigma_2$ be the half-twists supported on $\ell_1 \cup \ell_2$ and $\ell_2 \cup \ell_3$, then:
	\begin{align*}
		(\alpha,\beta,\gamma) \xmapsto{\sigma_1} (T_\alpha \beta, \alpha, \gamma).
	\end{align*}
	Now note that $i(T_\alpha \beta, \gamma) = 2$, and so \Cref{lemma:int2-inf} applies.
\end{proof}

\begin{lemma}\label{lemma:all-edge}
	Let $\pi$ be a genus one Lefschetz fibration over $D^2$ with $n$ nodal fibers such that $[B_n : \Br(\pi)] < \infty$. Then there is a complete Moishezon spider for $\pi$ with topological type $(\nu_1,\ldots,\nu_n)$ where $\nu_i \in \{\alpha,\beta\}$.
\end{lemma}

\begin{proof}
	Fix a complete Moishezon spider $(\nu_1,\ldots,\nu_n)$ for $\pi$. We will traverse the curves left to right. By \Cref{lemma:int2-inf}, we may assume that there are at least three distinct curves. Using the change of coordinates principle, the first curve $\nu_1$ can be sent to $\alpha$, and the second distinct curve $\nu_i \neq \nu_1$ can be identified with $\beta$. If there are no further curves, we are done.

	The third distinct curve (again moving left to right) must be $\nu_j = \de$, since otherwise the orbit would be infinite by \Cref{lemma:int2-inf} and \Cref{lemma:clockwise}. All curves $\nu_k$ for $k > j$ must have $\nu_k \in \{\de,\alpha\}$, since otherwise $\alpha, \de, \beta$ forms a counterclockwise triangle and the Hurwitz orbit is infinite by \Cref{lemma:clockwise}. By the same reasoning, for $i < k < j$ we must have that $\nu_k = \beta$. Thus, up to the action of $\SL_2\ZZ$,
	\begin{align*}
		(\nu_1,\ldots,\nu_n) = (\alpha,\ldots,\alpha,\beta,\ldots,\beta,\de,\nu_{j+1},\ldots,\nu_n),
	\end{align*}
	where $\nu_{j+1},\ldots,\nu_n \in \{\de,\alpha\}$. Applying a half-twist $\sigma$ to $(\alpha,\beta)$ yields
	\begin{align*}
		(\alpha,\beta) \xmapsto{\sigma} (T_\alpha \beta, \alpha).
	\end{align*}
	Direct computation then shows that $T_\alpha \beta = \de$. Thus, using Hurwitz moves we can transform our original spider into the form
	\begin{align*}
		(\alpha,\ldots,\alpha,\de,\ldots,\de,\alpha,\de,\nu_{j+1},\ldots,\nu_n).
	\end{align*}
	By the topological change of coordinates principle, there is a map in $\SL_2\ZZ$ fixing $\alpha$ and taking $\de$ to $\beta$. This completes the proof.
\end{proof}
Using these lemmas, we can now prove the theorem.

\begin{proof}[Proof of part \ref{item:virt-univ-class} of \Cref{thm:finite-classification}]
	Note that, for a given $n$, there are $2^n$ topological types of complete Moishezon spiders of the form given by \Cref{lemma:all-edge}. Without loss of generality, assume that the first curve is $\alpha$, bringing us down to $2^{n-1}$ topological types. By assuming the spider does not have type $(\alpha,\ldots,\alpha)$ we obtain $2^{n-1}-1$ types in total to check.

	For $n \leq 4$, if the Hurwitz orbit is finite then it is a finite check to verify that there exists a complete Moishezon spider of type $(\alpha,\beta)$, $(\alpha,\beta,\alpha)$, or $(\alpha,\beta,\alpha,\beta)$ in the Hurwitz orbit. We omit these calculations, and instead tabulate the infinite orbits.
	\begin{enumerate}
		\item[($n = 3$)] All orbits are finite.
		\item[($n = 4$)] There are two topological types with infinite orbit satisfying \Cref{lemma:all-edge}
			\begin{align*}
				(\alpha,\beta,\beta,\alpha) && (\alpha,\alpha,\beta,\beta).
			\end{align*}
			To show these orbits are infinite, apply Hurwitz moves to obtain
			\begin{align*}
				(\alpha,\beta,\beta,\alpha) &\xmapsto{\sigma_3} (\alpha,\beta,\gamma,\beta) \\
				(\alpha,\alpha,\beta,\beta) &\xmapsto{\sigma_2} (\alpha,\de,\alpha,\beta),
			\end{align*}
			and apply \Cref{lemma:clockwise}.
	\end{enumerate}
	We will show that for $n \geq 5$, all spiders except $(\alpha,\ldots,\alpha)$ correspond to genus one fibrations where $[B_n : \Br(\pi)] = \infty$. It suffices to do so for $n = 5$, where there are $2^4 - 1 = 15$ possibilities to check.
	\begin{enumerate}
		\item The tuples
			\begin{align*}
				&(\alpha,\alpha,\alpha,\beta,\beta) && (\alpha,\alpha,\beta,\beta,\alpha) && (\alpha,\alpha,\beta,\beta,\beta) && (\alpha,\beta,\alpha,\alpha,\beta) \\
				&(\alpha,\beta,\beta,\alpha,\alpha) && (\alpha,\beta,\beta,\beta,\alpha) && (\alpha,\alpha,\beta,\alpha,\beta) && (\alpha,\beta,\beta,\alpha,\beta)
			\end{align*}
			all contain a subspider of length $4$ of type $(\alpha,\beta,\beta,\alpha)$ or $(\alpha,\alpha,\beta,\beta)$, which gives rise to an infinite Hurwitz orbit. Recall that topological type is taken up to the action of $\SL_2\ZZ$, so that $(\beta,\alpha,\alpha,\beta)$ has type $(\alpha,\beta,\beta,\alpha)$ and $(\beta,\beta,\alpha,\alpha)$ has type $(\alpha,\alpha,\beta,\beta)$.
		\item Consider $(\alpha,\alpha,\alpha,\alpha,\beta)$. Apply the Hurwitz move $\sigma_2\sigma_3\sigma_4^{-1}$ to obtain
		   \begin{align*}
		   	(\alpha,\alpha,\alpha,\alpha,\beta) &\mapsto (\alpha,\alpha,\alpha,\beta,\de) \\
		   										&\mapsto (\alpha,\alpha,\de,\alpha,\de) \\
		   										&\mapsto (\alpha,\zeta, \alpha, \alpha, \de),
		   \end{align*}
		   where 
		   \begin{align*}
			   \alpha = \begin{pmatrix} 1 \\ 0 \end{pmatrix} && \beta = \begin{pmatrix} 0 \\ 1 \end{pmatrix} && \de = \begin{pmatrix} 1 \\ -1 \end{pmatrix} && \zeta = \begin{pmatrix} 2 \\ -1 \end{pmatrix}.
		   \end{align*}
		    The triple $(\alpha,\zeta,\de)$ is then a counterclockwise triangle and so this tuple has an infinite Hurwitz orbit by \Cref{lemma:clockwise}.
		\item Consider $(\alpha,\alpha,\alpha,\beta,\alpha)$. Apply the Hurwitz move $\sigma_3\sigma_4$ to obtain
		   \begin{align*}
		   	(\alpha,\alpha,\alpha,\beta,\alpha) &\mapsto (\alpha,\alpha,\beta,\alpha,\beta).
		   \end{align*}
		   The subspider $(\alpha,\alpha,\beta,\beta)$ gives rise to an infinite Hurwitz orbit by the calculation above for $n = 4$.
		\item Consider $(\alpha,\alpha,\beta,\alpha,\alpha)$, applying $\sigma_2\sigma_3$ produces a subspider of type $(\alpha,\beta,\beta,\alpha)$.
		\item Consider $(\alpha,\beta,\alpha,\alpha,\alpha)$, applying $\sigma_1\sigma_2$ produces a subspider of type $(\beta,\beta,\alpha,\alpha)$.
		\item Consider $(\alpha,\beta,\alpha,\beta,\alpha)$, applying $\sigma_2\sigma_1^2\sigma_4$ produces
		   \begin{align*}
		   	(\alpha,\beta,\alpha,\beta,\alpha) &\mapsto (\beta,\beta,\de,\gamma,\beta).
		   \end{align*}
		   The intersection number $i(\de,\gamma) = 2$ so by \Cref{lemma:int2-inf} we obtain an infinite Hurwitz orbit.
		\item Consider $(\alpha,\beta,\alpha,\beta,\beta)$. Applying $\sigma_1$ yields $(\alpha,\alpha,\beta,\beta)$ as a subspider.
		\item Consider $(\alpha,\beta,\beta,\beta,\beta)$. Applying $\sigma_2\sigma_1$ yields $(\de,\de,\beta,\beta)$ as a subspider.
	\end{enumerate}
	Equipped with this calculation, the proof is finished.
\end{proof}

\section{Branched covers associated to Lefschetz fibrations}\label{sec:branched-covers}

In this section we associate to a smooth Lefschetz fibration $\pi : M \to S$ a family of branched covers of $S$. Here and throughout the section we write $S$ for the base of the fibration, denoted $B$ in the introduction, reserving $B$ for the base of the abstract towers of covers in \Cref{lemma:abelian-lifting} and \Cref{lemma:double-lift}. Then in \Cref{subsec:sp-iso} we examine the level 2 branched cover associated to a genus one fibration and prove \Cref{thm:symp-iso-intro}. Let $\pi : M \to S$ be a smooth Lefschetz fibration of genus $g \geq 1$, with singular locus $\Delta \subseteq S$. In what follows $S$ is an arbitrary connected compact surface, possibly with $\partial S \neq \0$ and we write
\begin{align}
	\Mod_\pi \coloneqq \{[f] \in \Mod(S,\Delta) \st [f] \text{ admits a fiber-preserving lift to } M \text{ acting as the identity on } \pi^{-1}(\partial S)\},\label{eq:mod-pi}
\end{align}
to denote the group of liftable mapping classes instead of $\Br(\pi)$. Associated to any finite $\Mod(\Sigma_g)$-set $X$, one obtains via the Riemann realizability theorem a finite branched cover of $S$, via the monodromy representation
\begin{align}
	\phi_{\pi, X} : \pi_1(S \setminus \Delta) \to \Mod(\Sigma_g) \to \Perm(X),\label{defn:mon-rep-fin}
\end{align}
where $\Perm(X)$ denotes the group of permutations of $X$. Explicitly, the kernel $\ker \phi_{\pi, X}$ corresponds to a genuine cover of $S \setminus \Delta$ under the Galois correspondence, and one can fill in the punctures of the genuine cover to construct the desired branched cover.

Of particular interest is when $X$ is $\Mod(\Sigma_g)/H$ for some normal subgroup $H$ of $\Mod(\Sigma_g)$ with finite index. By a minor abuse of notation, we will denote the monodromy by 
\[
	\phi_{\pi_H} : \pi_1(S \setminus \Delta) \to \Mod(\Sigma_g)/H
\]
and the associated branched cover by 
\[
	\pi_H : \Sigma_\pi(H) \to S.
\]
Note that the branched cover is Galois.

We will attempt to analyze the fibration $\pi : M \to S$ through the above family of covers. In particular, we will focus on the following example.
\begin{example}\label{example:branched-congruence}
	Let $H$ be the congruence subgroup $\Mod(\Sigma_g)[N]$, i.e., the kernel of the symplectic representation modulo $N$, that is $\Mod(\Sigma_g) \twoheadrightarrow \Sp_{2g}(\ZZ/N\ZZ)$ (for properties of congruence subgroups, see \cite[\S 6.4]{primer}). The branched cover $\Sigma_\pi(H) \to S$ will be denoted in this case by 
	\begin{equation}
		\begin{aligned}
			\pi[N] &\coloneqq \pi_{\Mod(\Sigma_g)[N]}\\
			\Sigma_\pi[N] &\coloneqq \Sigma_\pi(\Mod(\Sigma_g)[N]) \\
			\pi[N] &: \Sigma_\pi[N] \to S. 
		\end{aligned}
		\label{defn:pi-N}
	\end{equation}
	Furthermore, since $\Mod(\Sigma_g)[N]$ is normal in $\Mod(\Sigma_g)$, the monodromy of the cover $\Sigma_\pi[N]$ is given by
	\begin{align}
		\phi_{\pi[N]} : \pi_1(S \setminus \Delta) \xrightarrow{\phi_\pi} \Mod(\Sigma_g) \to \Sp_{2g}(\ZZ/N\ZZ). \label{defn:mon-pi-N}
	\end{align}
	In particular, when $p$ is prime and $\phi : \pi_1(S \setminus \Delta) \to \Mod(\Sigma_g) \to \Sp_{2g}(\mathbb{F}_p)$ is surjective, $\Sigma_\pi[p] \to S$ is formed from $M \to S$ by taking the Galois closure of the cover formed by replacing each smooth fiber $F_x$ for $x \in S$ with $H^1(F_x;\mathbb{F}_p) - \{0\}$. The aforementioned Galois closure is equivalent to $\Sigma_\pi[p]$ because $\Sp_{2g}(\mathbb{F}_p)$ acts transitively on $(\mathbb{F}_p)^{2g} - \{0\}$ \cite[Proposition 3.2]{classical-groups-grove}, and hence the intermediate cover whose fibers are the nonzero vectors in $H_1(\Sigma_g; \mathbb{F}_p)$ is connected.
\end{example}

As discussed by Looijenga and by Putman--Wieland \cite{looijenga-prym-1997,putman-wieland}, any branched cover $p : \widetilde{S} \to S$ branched over $\Delta \subseteq S$ gives rise to a finite-index subgroup 
\begin{align*}
	\Mod_p \coloneqq \{[f] \in \Mod(S,\Delta) \st [f] \text{ admits a representative lifting to } \widetilde{S} \text{ acting as the identity on } p^{-1}(\partial S)\},
\end{align*}
as in \eqref{eq:mod-pi}. If $\partial S \neq \0$, the requirement that the lift acts as the identity on $p^{-1}(\partial S) = \partial \widetilde{S}$ allows us to identify $\Mod_p$ with a subgroup of $\Mod(\widetilde{S},p^{-1}(\Delta))$.

The following proposition relates the family of finite-index subgroups $\Mod_{\pi_H}$ of $\Mod(S,\Delta)$ with the (generally infinite-index) subgroup $\Mod_\pi$. In particular, when $S$ has one boundary component, it shows that the family of $\Mod_{\pi_H}$ determine $\Mod_\pi$. Throughout this section, we will explore the consequences of \Cref{prop:local-to-global-smooth} for certain congruence covers arising from elliptic fibrations.
\begin{proposition}\label{prop:local-to-global-smooth}
	Let $\pi : M \to S$ be a Lefschetz fibration of genus $g$ with singular locus $\Delta \subseteq S$ and let $I$ be a collection of finite-index subgroups of $\Mod(\Sigma_g)$ such that $\bigcap_{H \in I} H = 1$. Then
	\begin{align*}
		\Mod_\pi \subseteq \bigcap_{H \in I} \Mod_{\pi_H},
	\end{align*}
	with equality when $S$ has one boundary component. In particular, the result holds when $I$ is the collection of all finite-index subgroups of $\Mod(\Sigma_g)$.
\end{proposition}

\begin{proof}
	It suffices to prove the claim for any collection $I$ of finite-index subgroups with $\bigcap_{H \in I} H = 1$ since the claim for all finite-index subgroups then follows by the residual finiteness of $\Mod(\Sigma_g)$ \cite[Theorem 6.11]{primer}.
	Suppose first that $[f] \in \Mod_\pi$. As before, let $\phi_\pi$ be the monodromy representation of $\pi$. Then $\phi_\pi \circ
	f_\ast$ is conjugate to $\phi_\pi$ by an element of $\Mod(\Sigma_g)$, since $f \circ \pi$ and $\pi$ are isomorphic
	$\Sigma_g$-bundles over $S \setminus \Delta$. 
	Similarly, let $\phi_{\pi_H}$ be the monodromy representation of the cover $\pi_H$ as in \eqref{defn:mon-rep-fin}. It follows that $\phi_{\pi_H}$ is conjugate to $\phi_{\pi_H} \circ f_\ast$ for any subgroup $H$, proving the first inclusion.

	Now suppose $S$ has boundary, and choose a basepoint $b \in \partial S$.
	Supposing that $[f] \in \Mod_{\pi_H}$ for all $H \in
	I$ implies that $\phi_{\pi_H} \circ f_\ast =
	\phi_{\pi_H}$ for all $H \in I$. Then for any $\gamma \in
	\pi_1(S \setminus \Delta,b)$ we have
	\begin{align*}
		\phi_{\pi_H}(f_\ast(\gamma))\phi_{\pi_H}(\gamma)^{-1} = 1,
	\end{align*}
	and hence 
	\begin{align}
		\phi_\pi(f_\ast(\gamma))\phi_\pi(\gamma)^{-1} \in H.\label{eq:triv-local}
	\end{align}
	By assumption, since \eqref{eq:triv-local} holds for all $H$, we have that
	$\phi_\pi(f_\ast(\gamma)) = \phi_\pi(\gamma)$, and hence $\phi_\pi
	\circ f_\ast = \phi_\pi$. By a theorem of Moishezon when $g = 1$ and Matsumoto when $g \geq 2$ (see \cite[Theorem~2.1]{how-large}), since $\phi_\pi \circ f_\ast = \phi_\pi$ it follows that $f \in \Mod_\pi$ as desired.
\end{proof}

\begin{remark}
	As discussed in our previous paper \cite{how-large}, $\Mod_\pi$ generally has infinite index in $\Mod(S,\Delta)$, while $\Mod_{\pi_H}$ is always of finite index \cite{how-large}. With this in mind, we would like to draw attention to an analogy between $\Mod_\pi$ and the Torelli group, where each $\Mod_{\pi_H}$ plays the role of a congruence subgroup of $\Mod(S,\Delta)$. The analogy fails formally, since $\Mod_\pi$ is not normal in $\Mod(S,\Delta)$. However, it provides useful intuition as to the difficulty involved in the study of $\Mod_\pi$. See the discussion at the end of \Cref{subsec:sp-iso} for more on this analogy.
\end{remark}
Notably, \Cref{prop:local-to-global-smooth} provides a family of novel projective representations of $\Mod(\pi)$ and $\Mod_\pi$, as follows. Let $G$ be the deck group of $\Sigma_\pi(H) \to S$ and $\Sp_G(\Sigma_\pi(H))$ the group of symplectic transformations of $H^1(\Sigma_\pi(H))$ commuting with $G$, then as in \cite{looijenga-prym-1997} one obtains a \textit{higher Prym representation}
\begin{align*}
	\Mod_{\pi_H} \to \PSp_G(\Sigma_\pi(H)),
\end{align*}
where $\PSp_G(\Sigma_\pi(H))$ is the quotient of $\Sp_G(\Sigma_\pi(H))$ by its center.

The following subsections will consider the exceptional cases when $H$ is normal in $\Mod(\Sigma_g)$ and $\Mod(\Sigma_g)/H$ is solvable. Since $\Mod(\Sigma_g)$ is perfect for $g \geq 3$ \cite[Theorem 5.2]{primer}, such cases only occur when $g \leq 2$. The advantage of taking solvable quotients $\Mod(\Sigma_g)/H$ is the usual advantage: covers decompose as a tower of cyclic covers and these are particularly well understood.

\subsection{Branched covers of the sphere/disk arising from \texorpdfstring{$\SL_2 (\ZZ/N\ZZ)$}{SL\_2(Z/NZ)}}\label{subsec:sl2-mod}

Fix a Lefschetz fibration $\pi : M \to S$ of genus $g$ and singular locus $\Delta$. 
Consider $\Mod_{\pi[N]}$, where $\pi[N] : \Sigma_\pi[N] \to S$ is the branched cover corresponding to $\Mod(\Sigma_g)[N]$ as given in \Cref{example:branched-congruence}. 
For $N \geq 3$ let $\overline{\Sigma}_\pi[N] \to S$ denote the branched cover corresponding to the monodromy $\Mod(\Sigma_g) \to \Sp_{2g}(\ZZ/N\ZZ) \to \PSp_{2g}(\ZZ/N\ZZ)$.
Throughout this section we specialize to the case where $S = D^2$ and the fibration has $n$ singular fibers. Hence, we use the notation
\begin{align*}
	\Br(\pi) \coloneqq \Mod_\pi && \Br(\pi[N]) \coloneqq \Mod_{\pi[N]} < B_n
\end{align*}
as before. We record the following proposition relating $\Sigma_\pi[\ell]$ and $\overline{\Sigma}_\pi[\ell]$. It says that no information is lost in passing to the projective quotient, so that $\Br(\pi[\ell])$ may be read off the $\PSp$ tower.
\begin{proposition}\label{prop:psp-lifting}
	Let $\ell \geq 3$ be prime and let $\pi : M \to D^2$ be a Lefschetz fibration with $n$ singular fibers over $\Delta \subseteq D^2 \setminus \partial D^2$. Then any braid $\sigma$ in $B_n = \Mod(D^2, \Delta)$ which lifts to $\overline{\Sigma}_\pi[\ell]$ also lifts to $\Sigma_\pi[\ell]$.
\end{proposition}

\begin{proof}
	Let $\gamma_i$ for $1 \leq i \leq n$ be simple loops about the $n$ punctures $\Delta$, based at $b \in \partial D^2$.
	Also let $\phi : \pi_1(D^2 \setminus \Delta, b) \to \Sp_{2g}(\mathbb{F}_\ell)$ be the monodromy representation of $\pi[\ell]$. Fix some braid $\sigma$ lifting to $\overline{\Sigma}_\pi[\ell]$. Then for $\psi = \sigma \cdot \phi$
	\begin{align*}
		(\phi(\gamma_1),\ldots,\phi(\gamma_n)) &= (X_1,\ldots, X_n) \\
		(\psi(\gamma_1),\ldots,\psi(\gamma_n)) &= (\e_1 X_1,\ldots, \e_n X_n),
	\end{align*}
	for some $\e_i \in \mathbb{F}_\ell^\times$, since the center of $\Sp_{2g}(\mathbb{F}_\ell)$ is the set of scalar matrices $\e \Id$ with $\e^2 = 1$. Moreover, each $X_i$ and $\e_i X_i$ is a symplectic transvection (being conjugates of symplectic transvections) coming from a Dehn twist in $\Mod(\Sigma_g)$. Let $\tau_v(a) = a + \langle a,v \rangle v$ be the symplectic transvection about $v \in \mathbb{F}_\ell^{2g}$. Then for some nonzero $v,w \in \mathbb{F}_\ell^{2g}$ one has that $X_i = \tau_v$ and $\e_i X_i = \tau_w$, hence
	\begin{align*}
		\tau_w &= \e_i \tau_v \\
		w = \tau_w(w) &= \e_i\tau_v(w) = \e_i w + \e_i \langle w,v \rangle v,
	\end{align*}
	and hence either $\e_i = 1$ or $w$ is a scalar multiple of $v$. If $w$ is a scalar multiple of $v$, then
	\begin{align*}
		v = \tau_w(v) &= \e_i\tau_v(v) = \e_i v.
	\end{align*}
	Hence $\e_i = 1$. In any case, $\e_i = 1$ for all $i$, so $\psi = \phi$. Therefore $\sigma$ lifts to $\Sigma_\pi[\ell]$, completing the proof.
\end{proof}

We now restrict to genus one Lefschetz fibrations. It is well known that when $g = 1$ the mapping class group $\Mod(\Sigma_1)$ is isomorphic to $\SL_2\ZZ$ \cite[Section 2.2.4]{primer}. Furthermore, the quotients $\SL_2(\ZZ/N\ZZ)$ are solvable if and only if $N = 2^a3^b$ for some $a,b \in \ZZ_{\geq 0}$. Throughout this section let $\pi : M \to D^2$ be a genus one fibration with $n$ nodal fibers. Furthermore, fix some $b \in \partial D^2$ for convenience. Examples of such fibrations are given in \Cref{subsec:finite} and notated there as $q_n$ with monodromies
\begin{align*}
	(\phi_{q_n}(\gamma_1),\ldots,\phi_{q_n}(\gamma_n)) = (T_\alpha, T_\beta, T_\alpha, T_\beta, \ldots),
\end{align*}
where $(\alpha,\beta)$ is a geometric symplectic basis for $T^2$ and the $\gamma_i$ are simple loops about punctures in $D^2$.

We begin by exploring the special cases of $N = 2$ and $N = 3$, as these are particularly instructive. It is well known that $\PSL_2 \mathbb{F}_2 \cong \SL_2 \mathbb{F}_2 \cong S_3$ and $\PSL_2 \mathbb{F}_3 \cong A_4$, for example by examining their actions on $\PP^1(\mathbb{F}_2)$ and $\PP^1(\mathbb{F}_3)$ respectively. Hence, the covers $\Sigma_\pi[2]$ and $\Sigma_\pi[3]$ decompose as towers of branched covers
\begin{center}
	\begin{tikzcd}
		\Sigma_\pi[2] \ar[d,"\ZZ/3\ZZ"] \\
		S[2] \ar[d,"\ZZ/2\ZZ"] \\
		D^2
	\end{tikzcd} \hspace{0.1in}
	\begin{tikzcd}[column sep =small, row sep=tiny]
		\Sigma_\pi[3] \ar[dr,"\ZZ/2\ZZ"] &\\
		&\overline{\Sigma}_\pi[3] \ar[dd,"V_4"] \\
		\\
		&S[3] \ar[dl,"\ZZ/3\ZZ"] \\
		D^2,
	\end{tikzcd}
\end{center}
where $V_4 \cong \ZZ/2\ZZ \times \ZZ/2\ZZ$ is the Klein four group. We first identify these covers topologically, and find that both $S[2]$ and $S[3]$ are familiar from geometric topology.
\begin{proposition}\label{prop:mod2-hyper-super}
	Fix a nontrivial Lefschetz fibration $\pi : M \to D^2$ with $n$ singular fibers. Let $S[2]$ and $S[3]$ be the branched covers of $D^2$ induced by the monodromies
	\begin{align*}
		\pi_1(D^2 \setminus \Delta, b) &\to \Mod(T^2) \cong \SL_2\ZZ \to \SL_2\mathbb{F}_2 \to \ZZ/2\ZZ \\
		\pi_1(D^2 \setminus \Delta, b) &\to \Mod(T^2) \cong \SL_2\ZZ \to \SL_2\mathbb{F}_3 \to \ZZ/3\ZZ,
	\end{align*}
	respectively, with the maps above given by the decompositions of $\SL_2 \mathbb{F}_2 \cong S_3 \cong \ZZ/3\ZZ \rtimes \ZZ/2\ZZ$ and $\PSL_2 \mathbb{F}_3 \cong A_4 \cong V_4 \rtimes \ZZ/3\ZZ$. 

	Then $S[2]$ and $S[3]$ are the hyperelliptic and superelliptic covers of $D^2$, respectively. Furthermore, the covers $\Sigma_\pi[2] \to S[2]$ and $\Sigma_\pi[3] \to S[3]$ are unbranched. The hyperelliptic covers are depicted in \Cref{fig:hyperelliptic}. A standard reference for the hyperelliptic case is \cite[\S 9.4]{primer}.
\end{proposition}

\begin{proof}
	To prove the first claim it suffices to compute the monodromies, and note that they are equivalent to the well-known hyperelliptic and superelliptic monodromies. The monodromy about any simple loop $\gamma$ about a puncture is given by $T_\de$ for some vanishing cycle $\de$. Note that $T_\de$ is conjugate to $T_\alpha$ since any two simple closed curves in $T^2$ are equivalent under $\Mod(T^2) \cong \SL_2\ZZ$. Because $\ZZ/2\ZZ$ and $\ZZ/3\ZZ$ are abelian, it thus suffices to compute the monodromy for $T_\alpha = \begin{psmallmatrix} 1 & -1 \\ 0 & 1 \end{psmallmatrix} $. 

	Specializing to $N = 2$, the matrix $T_\alpha$ acts on $\PP^1(\mathbb{F}_2)$ by
	\begin{align*}
		[1 : 0] \mapsto [1 : 0] && [0 : 1] \mapsto [1 : 1] && [1 : 1] \mapsto [0 : 1].
	\end{align*}
	Thus, under the identification of $\SL_2\mathbb{F}_2$ with $S_3$, the Dehn twist $T_\alpha$ is sent to a transposition. Under the quotient $S_3 \to \ZZ/2\ZZ$, a transposition is taken to $1$, and so comparing with the monodromy of the hyperelliptic involution depicted in \Cref{fig:hyperelliptic} proves the result. When $N = 3$, acting on $\PP^1(\mathbb{F}_3)$ gives
	\begin{align*}
		[1 : 0] \mapsto [1 : 0] && [0 : 1] \mapsto [-1 : 1] \mapsto [1 : 1] \mapsto [0 : 1].
	\end{align*}
	Hence, identifying $\PSL_2 \mathbb{F}_3$ with $A_4 \subseteq S_4$, direct computation shows that $T_\alpha$ acts as a 3-cycle, and so maps to $\pm 1 \in \ZZ/3\ZZ$ as desired.

	To show that $\Sigma_\pi[2] \to S[2]$ and $\Sigma_\pi[3] \to S[3]$ are unbranched, it suffices to compute the total branching of $\Sigma_\pi[N]$ over each puncture in $D^2$ for $N \leq 3$ and show this quantity is equal to the branching for $S[N]$. The branching is computed by analyzing the cycle type of $T_\alpha$ acting on $\SL_2 \mathbb{F}_2$ and $\SL_2 \mathbb{F}_3$. We omit the calculation for the sake of brevity.
\end{proof}

\begin{figure}
	\centering
	\subfloat[$n$ odd]{
		\centering
		\includegraphics[scale=0.8]{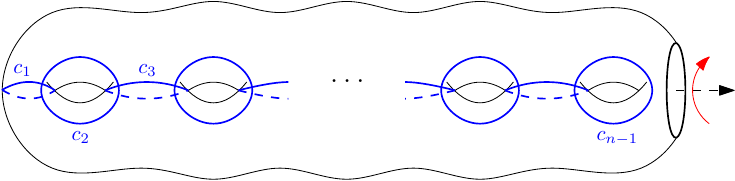}
	}
	\hspace{0.1in}
	\subfloat[$n$ even]{
		\centering
		\includegraphics[scale=0.8]{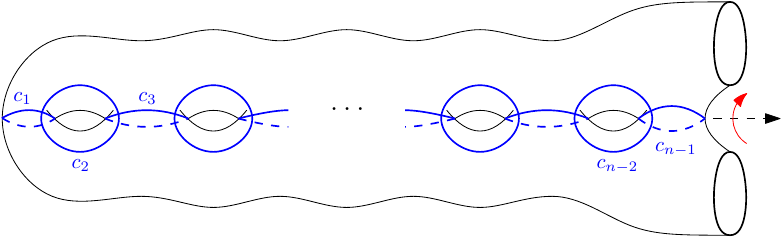}
	}\\
	\caption{The hyperelliptic covers branched over $n$ points on a disk for $n$ odd and $n$ even. The braid group $B_n$ acts by taking the half-twist $\sigma_i$ to the Dehn twist about $c_i$.}
	\label{fig:hyperelliptic}
\end{figure}

To identify $\Br(\pi[2])$ and $\Br(\pi[3])$ explicitly, we now recall a standard lemma concerning unbranched abelian covers, which can be found for example in \cite[\S 2]{looijenga-prym-1997} in slightly different language.
\begin{lemma}\label{lemma:abelian-lifting}
	Let $S$ be a connected, oriented surface with at most one boundary component, and let $p : \widetilde{S} \to S$ be an abelian cover with deck group $A$. Associated to the cover $p : \widetilde{S} \to S$ there is a cohomology class $\alpha \in H^1(S; A)$ such that the group of liftable mapping classes $\Mod_p$ is precisely
	\begin{align*}
		\Mod_p = \Stab_{\Mod(S)} \alpha < \Mod(S)
	\end{align*}
	where $\Mod(S)$ acts on $H^1(S;A)$ by mapping classes.

	Furthermore, when $\partial S \neq \0$, and we require mapping classes in $S,\widetilde{S}$ to be the identity on the boundary, we may identify $\Mod_p \subseteq \Mod(S)$ with a subgroup of $\Mod(\widetilde{S})$ that commutes with $A < \Mod(\widetilde{S})$. If $S$ does not have boundary, then there is a central extension
	\begin{align*}
		1 \to A \to \Mod_p^\# \to \Mod_p \to 1,
	\end{align*}
	where $\Mod_p^\# < \Mod(\widetilde{S})$ denotes the group consisting of all lifts of classes in $\Mod_p < \Mod(S)$.
\end{lemma}
\Cref{lemma:abelian-lifting} essentially follows by considering the action of $\Mod(S)$ on the monodromy representation $\pi_1(S) \to A$, which represents a cohomology class 
	\[
		\alpha \in H^1(S; A) \cong \Hom(H_1(S), A) \cong \Hom(\pi_1(S), A).
	\]
	Because the congruence covers $\pi[2] : \Sigma_\pi[2] \to D^2$ and $\pi[3] : \Sigma_\pi[3] \to D^2$ are expressed as towers of abelian covers, one must also understand the distinction between lifting for part of the tower and lifting the total cover when $S[2]$ or $S[3]$ has more than one boundary component.

In particular, consider a sequence of covers $p_2 : \widetilde{S} \to S, p_1 : S \to B$, where $B$ has one boundary component and $S$ has two boundary components. In this case, $[f] \in \Mod(B)$ can lift to $S$, fixing each boundary component $D_1,D_2 \subseteq \partial S$ and also lift to a fiber-wise diffeomorphism of $\widetilde{S}$ which fixes $p_2^{-1}(D_1)$. Nonetheless, the lift $\widetilde{f} : \widetilde{S} \to \widetilde{S}$ may permute the preimages $p_2^{-1}(D_2)$. The following lemma characterizes those $[f] \in \Mod_{p_1} < \Mod(B)$ that lift to $\widetilde{S}$ fixing the entire boundary $\partial \widetilde{S} = (p_1 \circ p_2)^{-1}(\partial B)$.
\begin{lemma}\label{lemma:double-lift}
	Let $B$ be a connected, compact, oriented surface with one boundary component, and let $p_1 : S \to B$ and $p_2 : \widetilde{S} \to S$ be connected abelian covers with deck groups $Q$ and $A$ respectively. 
	Then, fixing a basepoint $b \in \partial B$, the group $\Mod_{p_1}$ acts on $H^1(S, p_1^{-1}(b); A)$ by identifying $\Mod_{p_1} < \Mod(B)$ with the subgroup of lifts in $\Mod(S)$ by \Cref{lemma:abelian-lifting}. Associated to the tower of covers $\widetilde{S} \to S \to B$, there is a cohomology class $\alpha \in H^1(S, p_1^{-1}(b); A)$ such that 
	\begin{align*}
		\Mod_{p_1 \circ p_2} = \Stab_{\Mod_{p_1}} \alpha.
	\end{align*}
	That is, for any $[f] \in \Mod_{p_1}$, the mapping class $[f]$ lies in $\Mod_{p_1 \circ p_2}$ if and only if $F^\ast \alpha = \alpha$, where $F$ is the unique lift of $f$ to $S$ which fixes $\partial S = p_1^{-1}(\partial B)$. Furthermore, if $S$ has one boundary component, the relative cohomology can be replaced by $H^1(S; A)$.
\end{lemma}

\begin{proof}
	Let $P = p_1^{-1}(b)$ and fix a basepoint $\widetilde{b}_0 \in P$ as well as arcs $\ell_i$ from $\widetilde{b}_0$ to $\widetilde{b}_i \in P$ for each point in the fiber. Further, fix trivializations\footnote{The constructed class itself is not canonical, and in fact depends on these trivializations as well as the $\ell_i$.} of $p_2$ over each $\widetilde{b}_i$. We construct the cohomology class $\alpha \in H^1(S,p_1^{-1}(b); A)$ as a map out of $H_1(S, P; \ZZ)$ using the long exact sequence of a pair,
	\begin{align}
		0 \to H_1(S; \ZZ) \to H_1(S, P; \ZZ) \to \widetilde{H}_0(P;\ZZ) \to 0,\label{eq:h1-sequence}
	\end{align}
	which carries the additional structure of a sequence of $\Mod(S,P)$-representations. The restriction of $\alpha \in H^1(S,P;A) \cong \Hom(H_1(S,P;\ZZ), A)$ to $H_1(S; \ZZ)$ will be given by the monodromy representation of $\widetilde{S}$ over $S$. To extend to all of $H_1(S,P;\ZZ)$, we use the $\ell_i$ to provide a linear (i.e., not $\Mod(S,P)$-equivariant) splitting of the sequence \eqref{eq:h1-sequence}, and take the value of $\alpha$ on $[\ell_i] \in H_1(S,P;\ZZ)$ to be $0$.

	To prove the lemma, we first pass to the quotient $S/P$, and note that the trivializations over each $\widetilde{b}_i$ give us a way to identify points in each fiber via $(\widetilde{b}_i, a) \sim (\widetilde{b}_j, a)$ for all $a \in A$ so that $\widetilde{S}$ descends to an $A$-cover $\overline{S}$ over $S/P$. A self-map $F : S \to S$ so that $\rest{F}{P} = \Id$ then descends to a map $\overline{F} : S/P \to S/P$ taking the equivalence class $[P]$ to itself. The cohomology class above is then precisely represented by the monodromy of the cover $\overline{S}$ over $S/P$ with basepoint $[P]$, and so if $F^\ast \alpha = \alpha$ we have that $\overline{F}$ lifts to $\overline{S}$, acting as the identity on the fiber over $[P]$. Using the trivializations, $F$ lifts to $\widetilde{S}$, acting as the identity on the fiber over each $\widetilde{b}_i$. Applying the preceding argument in the case of $[f] \in \Mod_{p_1}$ and its lift $[F] \in \Mod(S)$ implies the result, since each boundary component of $\widetilde{S}$ is a circle with at least one marked point and any map over $B$ which fixes such a point can be isotoped over $B$ to one fixing the entire component. The converse is similar.

	We now must show that the cocycle $\alpha$ can be replaced by $\alpha' \in H^1(S; A)$, i.e. the restriction of $\alpha$ to $H_1(S;\ZZ)$, when $S$ has a single boundary component. To do so, choose each $\ell_i$ to live along the boundary $\partial S$. The choice of $\ell_i$ makes the linear splitting $\widetilde{H}_0(P; \ZZ) \to H_1(S,P; \ZZ)$ above a $\Mod(S, \partial S)$-equivariant splitting, and furthermore the subspace spanned by the $[\ell_i]$ is fixed pointwise. Hence the induced action on $H^1(S, P; A)$ is determined by the action on $H^1(S; A)$, proving the result.
\end{proof}

\begin{remark}
	As one should expect, \Cref{lemma:double-lift} applies (with appropriate modifications) even to the case of branched covers. In particular, if $\widetilde{S} \to S$ is unbranched, then \Cref{lemma:double-lift} is correct without modifications even if $S \to B$ is a branched cover.
\end{remark}

\subsection{Computation of the universal braid group for level 2 congruence covers}\label{subsec:sp-iso}

We now apply \Cref{lemma:abelian-lifting,lemma:double-lift} to the congruence covers $\pi[2]$ and $\pi[3]$ arising from a genus one Lefschetz fibration in order to prove \Cref{thm:symp-iso-intro}. We first prove the theorem when $\Br(\pi[2]) = B_n$ or the number of singular fibers $n$ is odd.

\begin{proof}[Proof of \Cref{thm:symp-iso-intro} when {$\Br(\pi[2]) = B_n$}]
	Let $\pi : M \to D^2$ be a genus one Lefschetz fibration with $n$ singular fibers and suppose that $\Br(\pi[2]) = B_n$. For convenience, let $g = \floor{\frac{n-1}{2}}$. Then for $\gamma_i,\gamma_{i+1}$ simple loops about the $i$-th, $i+1$-th branch points, the monodromies $\phi_{\pi[2]}(\gamma_i), \phi_{\pi[2]}(\gamma_{i+1})$ are transpositions in $\SL_2 \mathbb{F}_2$ by \Cref{prop:mod2-hyper-super}. Since $\sigma_i \in \Br(\pi[2])$, we have that
	\begin{align*}
		\phi_{\pi[2]}(\gamma_i) &= (\sigma_i \cdot \phi_{\pi[2]})(\gamma_i) \\
		\phi_{\pi[2]}(\gamma_i) &= \phi_{\pi[2]}(\gamma_{i+1})
	\end{align*}
	by direct computation. Because the action of $\sigma_i$ on $F_n$ consists of exchanging and conjugating the generators, $\Br(\pi[2]) = B_n$ if and only if $\phi_{\pi[2]}(\gamma_i) = \phi_{\pi[2]}(\gamma_j)$ for all $i,j$, and hence the monodromy lies entirely in a $\ZZ/2\ZZ$-subgroup of $\SL_2 \mathbb{F}_2$ as desired.
\end{proof}

\begin{proof}[Proof of \Cref{thm:symp-iso-intro} when $n$ is odd]
	Let $\pi : M \to D^2$ be a genus one Lefschetz fibration with $n$ singular fibers. For convenience, let $g = \floor{\frac{n-1}{2}}$. Assume that $n$ is odd and that $\Br(\pi[2]) \neq B_n$. Because $n$ is odd, $B_n$ surjects onto $\Sp(H^1(S[2]; \mathbb{F}_3)) \cong \Sp_{2g}(\mathbb{F}_3)$ by work of A'Campo \cite[Theorem 1]{acampo-tresses}. Since $S[2] \to D^2$ is the hyperelliptic cover, every braid lifts to $S[2]$, and hence applying \Cref{lemma:double-lift} shows also that
	\begin{align*}
		\Br(\pi[2]) = \Stab \alpha < B_n,
	\end{align*}
	where $\alpha \in H^1(S[2]; \mathbb{F}_3) \cong \mathbb{F}_3^{2g}$, since $S[2]$ has a single boundary component. The braid group $B_n$ acts through symplectic automorphisms on $H^1(S[2]; \mathbb{F}_3)$ because it acts on $S[2]$ by mapping classes. We can identify the action on $H^1(S[2];\mathbb{F}_3)$ with a representation $B_n \to \Sp_{2g}(\mathbb{F}_3)$. Because $\Br(\pi[2]) \neq B_n$, we know that $\alpha \neq 0$, and so
	\begin{align}
		\Core \Br(\pi[2]) = \bigcap_{\sigma \in B_n} \sigma\Stab \alpha \sigma^{-1} = \bigcap_{\sigma \in B_n} \Stab (\sigma \cdot \alpha) = \ker(B_n \to \Sp_{2g}(\mathbb{F}_3)),\label{eq:ker-calc}
	\end{align}
	where the last equality follows since $\Sp_{2g}(\mathbb{F}_3)$ acts transitively on $\mathbb{F}_3^{2g} \setminus \{0\}$. Equation \eqref{eq:ker-calc} implies the result, since
	\begin{align*}
		B_n/\Core(\Br(\pi[2])) = B_n/\ker(B_n \to \Sp_{2g}(\mathbb{F}_3)) \cong \Sp_{2g}(\mathbb{F}_3)
	\end{align*}
	by surjectivity. Note that the total monodromy is nontrivial, since $\phi_{\pi[2]}(\partial D^2)$ is the product of an odd number of transpositions.
\end{proof}

Before we begin the proof of \Cref{thm:symp-iso-intro} when the number of singular fibers $n$ is even, we recall some facts from the classical theory of Lie groups concerning the parabolic subgroups of $\Sp_{2h} k$ where $k$ is a field of characteristic not equal to $2$. For a standard reference see \cite{simple-groups-lie-type}, particularly Section 8.5. The following proposition is a consequence of the general theory contained therein.
\begin{proposition}[see {\cite[Section 8.5]{simple-groups-lie-type}}]\label{prop:levi}
	Let $(V,\omega)$ be a symplectic vector space over $\mathbb{F}_p$ where $p \neq 2$ is some prime. The stabilizer $(\Sp V)_v < \Sp V$ of $v \in V$ fits into an exact sequence
	\begin{align*}
		1 \to U \to (\Sp V)_v \to \Sp(v^\perp/v) \to 1,
	\end{align*}
	where $U$ is referred to as the unipotent radical. The unipotent radical satisfies the following:
	\begin{enumerate}
		\item $U$ is a central extension of the additive group $v^\perp/v$ by its center $Z(U)$:
	\begin{align}
		1 \to Z(U) \to U \to v^\perp/v \to 1.\label{eq:ext-U}
	\end{align}
		\item Concretely, $Z(U)$ consists of the symplectic transvections
			\begin{align*}
				Z(U) = \{L : a \mapsto a + c\omega(a,v)v \mid c \in \mathbb{F}_p\} \cong \mathbb{F}_p,
			\end{align*}
			and the map $U \to v^\perp/v$ in \eqref{eq:ext-U} is defined by 
			\begin{align*}
				U &\to v^\perp/v \\
				A &\mapsto A(w) - w
			\end{align*}
			for any basis element $w$ of $V/v^\perp$.
		\item The only $\Sp(v^\perp/v)$ conjugation invariant subgroups of $U$ are the trivial group, $Z(U)$, and $U$ itself.
	\end{enumerate}
\end{proposition}

\begin{proof}
	Let $H$ be a subgroup of $U$ invariant under $(\Sp V)_v$, and let $\phi : U \to v^\perp/v$ be the group homomorphism inducing \eqref{eq:ext-U}. The group $\phi(H) \subseteq v^\perp/v$ is invariant under $(\Sp V)_v$ and hence under $\Sp(v^\perp/v)$. Since $\Sp(v^\perp/v)$ acts irreducibly on $v^\perp/v$, either $\phi(H) = 1$ or $\phi(H) = v^\perp/v$.

	If $\phi(H) = 1$, then $H \subseteq Z(U) \cong \mathbb{F}_p$. Hence $H = 1$ or $H = Z(U)$. Otherwise, $\phi(H) = v^\perp/v$. If $H \cap Z(U) \neq 1$, then $H = U$ and we are done. If $H \cap Z(U) = 1$, then $H$ provides a splitting $\psi : v^\perp/v \to U$ of \eqref{eq:ext-U}. We prove no such splitting can exist.

	Let $x,y,w,v$ form part of a symplectic basis for $V$ where $\omega(x,y) = \omega(v,w) = 1$ and all other pairings vanish. We define $M(x),M(y) \in \Sp V$ by
	\begin{align*}
		M(x) &: x\mapsto x, y \mapsto y+v, w \mapsto w + x, v\mapsto v, \\
		M(y) &: x \mapsto x-v, y \mapsto y,  w \mapsto w + y, v\mapsto v,
	\end{align*}
	where both act as the identity on the remaining vectors of the symplectic basis. A direct check shows that $M(x),M(y) \in U$ and project to $x,y$ respectively. We compute that
	\begin{align*}
		[M(x),M(y)](w) &= M(x)M(y)(M(x))^{-1} \cdot (w-y)  \\
					   &= M(x)M(y)\cdot (w-x - y + v) \\
					   &= M(x) \cdot (w - x + 2v) \\
					   &= w + 2v.
	\end{align*}
	Now since $\psi$ is a group homomorphism and $v^\perp/v$ is abelian we see that $[\psi(x),\psi(y)] = \Id$. But $[\psi(x),\psi(y)]$ is also equal to $[M(x),M(y)]$ since $M(x) = \psi(x)z$ for some $z \in Z(U)$ and $M(y) = \psi(y)z'$ for some $z' \in Z(U)$. Thus we have a contradiction as $2 \neq 0$ in $\mathbb{F}_p$ by assumption.
\end{proof}

Equipped with \Cref{prop:levi}, we prove the surjectivity portion of \Cref{thm:symp-iso-intro} when the number of singular fibers $n$ is even.
\begin{lemma}\label{lemma:surj-symp}
	Suppose $\pi : M \to D^2$ is a genus one fibration with $n$ singular fibers where $n \geq 4$ is even. Furthermore fix some $b \in \partial D^2$ and let $p : S[2] \to D^2$ be the induced hyperelliptic cover as in \Cref{prop:mod2-hyper-super} and \Cref{fig:hyperelliptic}. Then the induced representation
	\begin{align*}
		\rho : B_n \to \Sp H^1(S[2], p^{-1}(b); \mathbb{F}_\ell)
	\end{align*}
	surjects onto the stabilizer of a vector for every prime $\ell \neq 2$.
\end{lemma}

\begin{figure}[t!]
	\centering
	\includegraphics{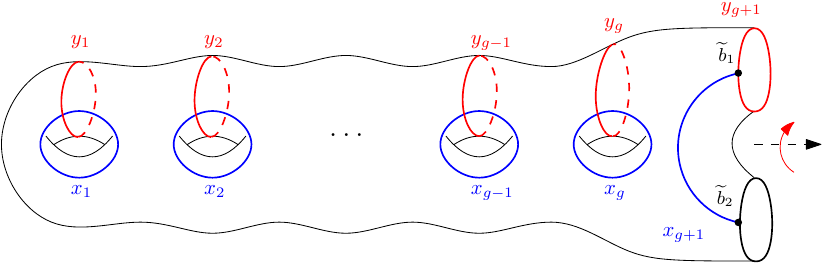}
	\caption{A symplectic basis for $H^1(S[2], \{\widetilde{b}_1,\widetilde{b}_2\})$ when $n$ is even.}
	\label{fig:symp-basis}
\end{figure}

\begin{proof}
	For convenience, let $V = H^1(S[2],p^{-1}(b); \mathbb{F}_\ell)$. Choose a basis $x_1,y_1,\ldots,x_{g+1},y_{g+1}$ where $g = \frac{n-2}{2}$ for $V$ as given in \Cref{fig:symp-basis}. The half-twist generators of $B_n$ lift to explicit Dehn twists $T_{c_i}$ as depicted in \Cref{fig:hyperelliptic}. Hence, the image of $B_n$ under the representation above lies in the stabilizer $(\Sp V)_{y_{g+1}}$.

	By work of A'Campo \cite[Theorem 1]{acampo-tresses}, the composition
	\begin{align*}
		B_n \to (\Sp V)_{y_{g+1}} \to \Sp(y_{g+1}^\perp/y_{g+1})
	\end{align*}
	is surjective. Hence, it suffices to show that $U = \im \rho \cap U$, where $U$ is the unipotent radical referenced in \Cref{prop:levi}. 
	
	Because $B_n$ surjects onto $\Sp(y_{g+1}^\perp/y_{g+1})$, the subgroup $\im \rho \cap U$ is invariant under the action of $\Sp(y_{g+1}^\perp/y_{g+1})$ coming from the exact sequence 
	\begin{align}
		1 \to U \to (\Sp V)_{y_{g+1}} \to \Sp(y_{g+1}^\perp/y_{g+1}) \to 1. \label{eq:sp-stab-sequence}
	\end{align}
	As proved in \Cref{prop:levi}, the only such invariant subgroups are $1,Z(U)$, and $U$ itself. Hence, to prove surjectivity onto $U$ it suffices to find an element of $\im \rho \cap U$ which is non-central. The braid $\sigma_{n-1}$ acts by Dehn twist about the curve $c_{n-1}$ in \Cref{fig:hyperelliptic} and acts on $H^1$ by
	\begin{align*}
		x_{g+1} &\mapsto x_{g+1} - c_{n-1} = x_{g+1} - y_{g+1} + y_g \\
		x_g &\mapsto x_g + y_{g+1} - y_g,
	\end{align*}
	and by the identity in the other basis vectors. Let $W$ denote the span of $x_1,y_1,\ldots,x_g,y_g$. Since $n$ is even, the braid group $B_{n-1}$ supported on the first $n-1$ strands surjects onto $\Sp(W)$ by A'Campo's work. Therefore, there is some braid $\sigma \in B_{n-1}$ which acts on $W$ by
	\begin{align*}
		x_i &\mapsto x_i \tag*{$\forall 1 \leq i \leq g-1$} \\
		y_i &\mapsto y_i \tag*{$\forall 1 \leq i \leq g-1$} \\
		x_g &\mapsto x_g + y_g \\
		y_g &\mapsto y_g
	\end{align*}
	Thus we obtain a braid $f = \sigma \sigma_{n-1}$ so that $\rho(f)$ acts by
	\begin{align*}
		x_{g+1} &\xmapsto{\rho(f)} x_{g+1} - y_{g+1} + y_g \\
		x_g &\xmapsto{\rho(f)} x_g + y_{g+1} 
	\end{align*}
	while fixing the other basis vectors. Since $x_1,y_1,\ldots,x_g,y_g$ project to a basis for $y_{g+1}^\perp/y_{g+1}$, we find that $\rho(f) \in U$. Furthermore, since 
	\[
		\rho(f) \cdot x_{g+1} -x_{g+1} = y_g - y_{g+1}
	\]
	is nontrivial in $y_{g+1}^\perp/y_{g+1}$, we see that $\rho(f) \not\in Z(U)$ by \Cref{prop:levi}.

	Thus $\rho(f)$ lies in $\im \rho \cap U$ and is non-central, as claimed, and so $\im \rho \cap U = U$ as desired. By the exact sequence \eqref{eq:sp-stab-sequence}, $\rho$ surjects onto $(\Sp V)_{y_{g+1}}$ as desired.
\end{proof}

We now are left to prove the injectivity portion of \Cref{thm:symp-iso-intro}. We start with a lemma which applies regardless of the boundary monodromy.
\begin{lemma}\label{lemma:injective-symp}
	Suppose $\pi : M \to D^2$ is a genus one fibration with $n$ singular fibers where $n \geq 4$ is even such that $\Br(\pi[2]) \neq B_n$. Furthermore fix some $b \in \partial D^2$ and let $p : S[2] \to D^2$ be the induced hyperelliptic cover as in \Cref{prop:mod2-hyper-super} and \Cref{fig:hyperelliptic}. Let $\rho$ denote the induced representation
	\begin{align*}
		\rho : B_n \to \Sp H^1(S[2], p^{-1}(b); \mathbb{F}_3),
	\end{align*}
	then $\rho(\Core(\Br(\pi[2]))) \subseteq Z(U)$, where $U$ is the unipotent radical of \Cref{prop:levi}.
\end{lemma}

\begin{proof}
	Let $V \coloneqq H^1(S[2],p^{-1}(b); \mathbb{F}_3)$ for convenience, with the symplectic form represented by $\omega$. As before choose the symplectic basis $x_1,y_1,\ldots,x_{g+1},y_{g+1}$ for $V$ given in \Cref{fig:symp-basis} where $g = \frac{n-2}{2}$. By \Cref{lemma:double-lift}, $\Br(\pi[2]) = \Stab\, [\alpha]$ for some $[\alpha] \in H^1(S[2],p^{-1}(b); \mathbb{F}_3)$. Furthermore, by the construction of $\alpha$ presented in \Cref{lemma:double-lift}, $\alpha$ must have some nontrivial $x_i$ or $y_i$ component for $1 \leq i \leq g$. Otherwise, the monodromy about each puncture of $D^2$ would map to the same transposition in $\SL_2 \mathbb{F}_2 \cong S_3$, and so every simple braid would lift. However, by assumption $\Br(\pi[2]) \neq B_n$.

	Furthermore, the construction of \Cref{lemma:double-lift} allows us to take $\alpha$ to have trivial pairing with $x_{g+1}$. Furthermore if \[
		\phi_{\pi[2]} : \pi_1(D^2 \setminus n \text{ points}) \to \SL_2 \mathbb{F}_2
	\]
	is the monodromy of $\pi[2]$, then the total monodromy $\phi_{\pi[2]}(\partial D^2)$ is a 3-cycle in $\SL_2 \mathbb{F}_2 \cong S_3$. The element of $\ZZ/3\ZZ$ induced by a choice of embedding $\ZZ/3\ZZ \to S_3$ can be identified with $\omega(\alpha,y_{g+1})$. Therefore,
	\begin{align*}
		\alpha = a_{g+1}x_{g+1} + \sum_{i=1}^g (a_ix_i + b_iy_i),
	\end{align*}
	where $a_{g+1} = 0$ if and only if the total monodromy $\phi_{\pi[2]}(\partial D^2)$ is trivial, and at least one of $a_i, b_i \neq 0$ for some $1 \leq i \leq g$.

	By the above notes $[\alpha] \neq 0$ in $y_{g+1}^\perp/y_{g+1}$, since one of $a_i,b_i \neq 0$. Because $B_n$ surjects onto $\Sp(y_{g+1}^\perp/y_{g+1})$ and 
	\begin{align*}
		\rho(\Core(\Br(\pi[2]))) \subseteq \Stab\, [\sigma \cdot \alpha]
	\end{align*}
	for all $\sigma \in B_n$, we must then have that $\rho(\Core(\Br(\pi[2])))$ acts trivially on $y_{g+1}^\perp/y_{g+1}$. Namely, $\Sp(y_{g+1}^\perp/y_{g+1})$ acts transitively on nonzero vectors \cite[Proposition 3.2]{classical-groups-grove}. Hence $\rho(\Core(\Br(\pi[2]))) \subseteq U$.

	As before $\rho(\Core(\Br(\pi[2])))$ is $\Sp(y_{g+1}^\perp/y_{g+1})$ invariant, and so it suffices by \Cref{prop:levi} to show it cannot be all of $U$. Consider the map $M(w)$ defined for $w$ in the span of $x_1,y_1,\ldots,x_g,y_g$ by
	\begin{align*}
		M(w) : x_i &\mapsto x_i + \omega(x_i,w)y_{g+1} \tag{$1 \leq i \leq g$} \\
		x_{g+1} &\mapsto x_{g+1} + w  \\
		y_i &\mapsto y_i + \omega(y_i,w)y_{g+1} \tag{$1 \leq i \leq g$} \\
		y_{g+1} &\mapsto y_{g+1},
	\end{align*}
	where $\omega$ is the symplectic pairing. Direct computation shows that $M(w) \in U$ for all $w$ in the span of $x_1,y_1,\ldots,x_g,y_g$. Furthermore,
	\begin{align*}
		M(x_i)(\alpha) &= \alpha + a_{g+1}x_i - b_iy_{g+1}, \\
		M(y_i)(\alpha) &= \alpha + a_{g+1}y_i + a_iy_{g+1}.
	\end{align*}
	Since one of $a_i,b_i \neq 0$, at least one of these does not stabilize $\alpha$ and hence does not lie in $\rho(\Core(\Br(\pi[2])))$, which implies that $\rho(\Core(\Br(\pi[2]))) \subseteq Z(U)$.
\end{proof}

Using \Cref{lemma:surj-symp} and \Cref{lemma:injective-symp}, we finish the proof of \Cref{thm:symp-iso-intro}.

\begin{proof}[Proof of {\Cref{thm:symp-iso-intro}} when $n$ is even and {$\Br(\pi[2]) \neq B_n$}]
	Suppose $\pi : M \to D^2$ is a genus one fibration with $n$ singular fibers where $n$ is even and $\Br(\pi[2]) \neq B_n$. When $n = 2$ the result is immediate: $B_2 \cong \ZZ$, the total monodromy $\phi_{\pi[2]}(\partial D^2)$ is a product of two distinct transpositions and hence a $3$-cycle, and $(\Sp_2 \mathbb{F}_3)_{y_1} \cong \ZZ/3\ZZ$. We therefore assume $n \geq 4$. Furthermore fix some $b \in \partial D^2$ and let $p : S[2] \to D^2$ be the induced hyperelliptic cover as in \Cref{prop:mod2-hyper-super} and \Cref{fig:hyperelliptic}. Let $\rho$ denote the induced representation
	\begin{align*}
		\rho : B_n \to \Sp H^1(S[2], p^{-1}(b); \mathbb{F}_3)
	\end{align*}
	and let 
	\[
		\phi_{\pi[2]} : \pi_1(D^2 \setminus n \text{ points}) \to \SL_2 \mathbb{F}_2
	\]
	denote the monodromy of $\pi[2] : \Sigma_\pi[2] \to D^2$. As before let $V = H^1(S[2], p^{-1}(b); \mathbb{F}_3)$. By \Cref{lemma:surj-symp}
	\begin{align*}
		B_n/\ker(\rho) \cong \im \rho = (\Sp V)_{y_{g+1}}
	\end{align*}
	where $y_{g+1}$ is depicted in \Cref{fig:symp-basis} as part of a symplectic basis $x_1,y_1,\ldots,x_{g+1},y_{g+1}$ for $V$. Furthermore, there is some vector $\alpha \in V$ so that $\Br(\pi[2]) = \Stab\, [\alpha]$ and
	\begin{align*}
		\alpha = a_{g+1}x_{g+1} + \sum_{i=1}^g (a_ix_i + b_iy_i),
	\end{align*}
	in the basis $x_1,y_1,\ldots,x_{g+1},y_{g+1}$ given in \Cref{fig:symp-basis}. As discussed in the proof of \Cref{lemma:surj-symp}, $a_{g+1} = 0$ if and only if the total monodromy $\phi_{\pi[2]}(\partial D^2)$ is trivial.

	By \Cref{lemma:injective-symp}, we know that $\rho(\Core(\Br(\pi[2]))) \subseteq Z(U)$, where
	\begin{align*}
		1 \to U \to (\Sp V)_{y_{g+1}} \to \Sp y_{g+1}^\perp/y_{g+1} \to 1
	\end{align*}
	is the Levi decomposition. By \Cref{prop:levi} the image $\rho(\Core(\Br(\pi[2])))$ is $Z(U)$ or trivial. To prove \Cref{thm:symp-iso-intro}, it suffices to determine which group occurs depending on the value of $a_{g+1}$. 

	Note that the element of $Z(U) \cong \mathbb{F}_3$ associated to $\lambda \in \mathbb{F}_3$ acts on $\alpha$ by
	\[
		\alpha \mapsto \alpha + a_{g+1}\lambda y_{g+1}.
	\] 
	Therefore if $a_{g+1} \neq 0$ no nontrivial element of $Z(U)$ can belong to $\rho(\Br(\pi[2]))$, and so $\rho(\Core(\Br(\pi[2]))) = 1$. In other words, if the monodromy $\phi_{\pi[2]}(\partial D^2)$ is nontrivial, then $\ker \rho = \Core(\Br(\pi[2]))$ and so
	\begin{align*}
		B_n/\Core(\Br(\pi[2])) = (\Sp V)_{y_{g+1}},
	\end{align*}
	as claimed in \Cref{thm:symp-iso-intro}.

	Finally, suppose that $a_{g+1} = 0$. Since $B_n$ surjects onto $(\Sp V)_{y_{g+1}}$ by \Cref{lemma:surj-symp}, there is some $\sigma \in B_n$ mapping to the central shear 
	\[
		x_{g+1} \mapsto x_{g+1} + y_{g+1}
	\] 
	in $Z(U)$. We claim that $\sigma$ lies in $\Core(\Br(\pi[2]))$. We compute that $\rho(\sigma)(\alpha) = \alpha$ since $a_{g+1} = 0$, and so $\sigma \in \Br(\pi[2])$. Furthermore, $\rho(\sigma)$ lies in the center of $(\Sp V)_{y_{g+1}}$, and hence $\rho(\sigma'\sigma\sigma'^{-1}) = \rho(\sigma)$ for each $\sigma' \in B_n$. Thus the conjugates also fix $\alpha$, and so lie in $\Br(\pi[2])$. As a consequence $\rho(\Core(\Br(\pi[2]))) = Z(U)$ and $\ker \rho \subseteq \Core(\Br(\pi[2]))$ which implies
	\begin{align*}
		B_n/\Core(\Br(\pi[2])) \cong \im \rho/\rho(\Core(\Br(\pi[2]))) = (\Sp V)_{y_{g+1}}/Z(U)
	\end{align*}
	as claimed in \Cref{thm:symp-iso-intro}. With the above calculations, \Cref{thm:symp-iso-intro} is proved. The connection to the integral Burau representation follows from the fact that the action of $B_n$ on the hyperelliptic cover is precisely the integral Burau representation (see \cite{brendle-margalit-putman-hyperelliptic-torelli}).
\end{proof}

\begin{remark}\label{rem:potential-mod-3}
	Using \Cref{lemma:double-lift}, a similar computation should be possible for $\Core(\Br(\pi[3]))$. In this case, the presence of superelliptic covers would imply a connection to the Burau representation evaluated at $t$ a cube root of unity $\omega$. In particular, the key technical input would be the calculation of the image
	\begin{align*}
		B_n \to \GL_{n-1}(\ZZ[t,t^{-1}]) \to \GL_{n-1}(\ZZ[\omega]) \to \GL_{n-1}(\ZZ[\omega]/(2)) = \GL_{n-1}(\mathbb{F}_4).
	\end{align*}
	As shown by Squier, the representation above preserves an explicit hermitian form \cite{squier-unitary}.
\end{remark}
\Cref{thm:symp-iso-intro} partially justifies an analogy between $\Br(\pi)$ and the Torelli group, as $\Core(\Br(\pi[2]))$ is the kernel of a particular linear representation. One may take the analogy even further. Let $\pi : M \to D^2$ be a genus one fibration with nodal fibers. Repeated applications of \Cref{lemma:abelian-lifting} to the system of branched covers $\pi[2^a]$ show that $\Br(\pi) = \bigcap_a \Br(\pi[2^a])$ is determined by a sequence of cohomological representations of $\Br(\pi[2^{a-1}])$ over finite fields. \textit{A posteriori}, the simple fact that $\SL_2\ZZ$ is residually solvable implies the existence of such a linear interpretation of $\Br(\pi)$ when the base is a disk. We end by asking the following question.
\begin{question}\label{q:torelli-like}
	Let $\pi : M \to S^2$ be a genus one elliptic fibration. Is $\Br(\pi)$ equal to $\bigcap_N \Br(\pi[N])$?
\end{question}
\Cref{q:torelli-like} is number-theoretic in nature. Indeed, it is equivalent to a local-global principle for the conjugacy problem of representations $\pi_1(S^2 \setminus \Delta) \to \SL_2\ZZ$ satisfying a strong parabolicity condition---that each simple closed curve about a puncture is mapped to a conjugate of $\begin{psmallmatrix} 1 & - 1 \\ 0 & 1 \end{psmallmatrix}$. Dropping this condition and instead requiring that each generator simply has trace $\pm 2$, one can easily construct counterexamples.

\printbibliography

\end{document}